\def\TITLE{\bf Regularity of quantum $\tau$-functions
generated by quantum birational Weyl group actions}
\def\AUTHOR{Gen Kuroki\thanks{\ADDRESS.}}
\def\DATE{June 15, 2012}
\def\ADDRESS{Mathematical Institute, Tohoku University, Sendai 980-8578, Japan}
\def\ABSTRACT{
We canonically quantize the $\tau$-functions for 
the birational Weyl group action arising from a nilpotent
Poisson algebra proposed by Noumi and Yamada.
We also construct the $q$-difference deformation 
of the canonical quantization of the $\tau$-functions. 
Using the translation functors for the symmetrizable Kac-Moody algebras,
we prove the regularity of the quantum $\tau$-functions,
namely, we show that the quantum $\tau$-functions   
are polynomials in dependent variables.
}
\def\ACKNOWLEDGEMENTS{
The author would like to thank Koji Hasegawa for valuable discussions
on the topics of this paper.
This work was supported by Grant-in-Aid for Scientific Research No.~23540003.
}
\documentclass[12pt,twoside]{article}
\usepackage{amsmath,amssymb,amsthm}
\usepackage{hyperref}
\newcommand\arxivref[1]{\href{http://arxiv.org/abs/#1}{\tt arXiv:#1}}
\newcommand\bra{\langle}
\newcommand\ket{\rangle}
\newcommand\Hom{\mathop{\mathrm{Hom}}\nolimits}

\newcommand\A{{\mathcal A}}

\renewcommand\O{{\mathcal O}}

\makeatletter\newcommand\qbinom{\genfrac[]\z@{}}\makeatother
\newcommand\ad{\mathop{\mathrm{ad}}\nolimits}
\newcommand\Ad{\mathop{\mathrm{Ad}}\nolimits}
\newcommand\av{\alpha^\vee}
\newcommand\eps{\varepsilon}
\newcommand\epsv{\eps^\vee}
\newcommand\deltav{\delta^\vee}
\newcommand\Qv{Q^\vee}
\newcommand\tW{{\widetilde W}}
\newcommand\tw{{\widetilde w}}
\newcommand\ts{{\tilde s}}
\newcommand\tU{{\widetilde U}}
\newcommand\tA{{\widetilde A}}

\newcommand\pr{\mathop{\mathrm{pr}}\nolimits}

\newcommand\g{{\mathfrak g}}
\newcommand\h{{\mathfrak h}}
\newcommand\n{{\mathfrak n}}

\renewcommand\setminus{\smallsetminus}
\renewcommand\d{\partial}
\newcommand\pa{{\mathrm{pa}}}
\newcommand\bs{{\mathbf s}}
\newcommand\be{{\mathbf e}}
\newcommand\intpart{P}
\newcommand\Oint{\O_\intpart}
\newcommand\Ointh{\O_{\hbar,\intpart}}
\newcommand\Ointg{\Oint^{\mathrm{g}}}

\newcommand\Kint{K_{\intpart}}
\newcommand\Kintg{\Kint^{\mathrm{g}}}
\newcommand\Ob{\mathop{\mathrm{Ob}}\nolimits}
\newcommand\Peq[1]{\mathrm{P}_{\mathrm{#1}}}

\newcommand\PII{\Peq{II}}
\newcommand\PIII{\Peq{III}}
\newcommand\PIV{\Peq{IV}}
\newcommand\PV{\Peq{V}}
\newcommand\PVI{\Peq{VI}}

\newcommand\Z{{\mathbb Z}} 
\newcommand\C{{\mathbb C}} 
\theoremstyle{plain} 
\newtheorem{theorem}{Theorem}
\newtheorem*{theorem*}{Theorem}
\newtheorem{prop}[theorem]{Proposition}
\newtheorem*{prop*}{Proposition}
\newtheorem{lemma}[theorem]{Lemma}
\newtheorem*{lemma*}{Lemma}

\newtheorem*{cor*}{Corollary}

\newtheorem*{axiom*}{Axiom}

\newtheorem*{problem*}{Problem}

\newtheorem*{summary*}{Summary}

\newtheorem*{guide*}{Guide}
\theoremstyle{definition} 
\newtheorem{definition}[theorem]{Definition}
\newtheorem*{definition*}{Definition}
\theoremstyle{definition} 
\newtheorem{remark}[theorem]{Remark}
\newtheorem*{remark*}{Remark}
\newtheorem{example}[theorem]{Example}
\newtheorem*{example*}{Example}
\numberwithin{theorem}{section}
\numberwithin{equation}{section}
\numberwithin{figure}{section}
\numberwithin{table}{section}
\newcommand\secref[1]{Section \ref{#1}}
\newcommand\theoremref[1]{Theorem \ref{#1}}
\newcommand\propref[1]{Proposition \ref{#1}}
\newcommand\lemmaref[1]{Lemma \ref{#1}}

\newcommand\exampleref[1]{Example \ref{#1}}

\newcommand\definitionref[1]{Definition \ref{#1}}
\newcommand\remarkref[1]{Remark \ref{#1}}
\makeatletter
\renewenvironment{proof}[1][\proofname]{\par
  \normalfont
  \topsep6\p@\@plus6\p@ \trivlist
  \item[\hskip\labelsep{\bfseries #1}\@addpunct{\bfseries.}]\ignorespaces
}{%
  \endtrivlist
}
\renewcommand{\proofname}{Proof}
\makeatother
\makeatletter
\def\BOXSYMBOL{\RIfM@\bgroup\else$\bgroup\aftergroup$\fi
  \vcenter{\hrule\hbox{\vrule height.85em\kern.6em\vrule}\hrule}\egroup}
\makeatother
\newcommand{\BOX}{%
  \ifmmode\else\leavevmode\unskip\penalty9999\hbox{}\nobreak\hfill\fi
  \quad\hbox{\BOXSYMBOL}}
\renewcommand\qed{\BOX}
\begin{document}
\title{\TITLE}
\author{\AUTHOR}
\date{\DATE}
\maketitle
\begin{abstract}
  \ABSTRACT
\end{abstract}
\tableofcontents
\setcounter{section}{-1} 

\section{Introduction}


In the previous paper \cite{Kuroki2008}, the author canonically quantized 
the birational Weyl group action arising from a nilpotent
Poisson algebra proposed by Noumi and Yamada \cite{NY0012028}.
But he did not quantize their $\tau$-functions.
In this paper, we shall quantize the $\tau$-functions and
prove that the quantized $\tau$-functions are polynomials in dependent variables.

Let $[a_{ij}]_{i,j\in I}$ be any symmetrizable generalized Cartan matrix
(GCM for short) with positive integers $d_i$ ($i\in I$) satisfying $d_i a_{ij}=d_j a_{ji}$.
Denote by $W$ the Weyl group of the GCM $[a_{ij}]_{i,j\in I}$
generated by the simple reflections $s_i$ ($i\in I$).


\subsection{Classical case}

Following Noumi and Yamada \cite{NY0012028}, we define a nilpotent Poisson algebra 
to be a Poisson commutative integral domain generated by $\{f_i\}_{i\in I}$ 
as a Poisson algebra with the following nilpotency property of the Poisson bracket:
\begin{equation*}
 (\ad_{\{,\}} f_i)^{1-a_{ij}}(f_j)
 = \underbrace{\{f_i,\{\cdots,\{f_i,\{f_i}_{\text{$1-a_{ij}$ times}},f_j\}\}\cdots\}\}
 = 0
 \quad (i\ne j),
\end{equation*}
where $(\ad_{\{,\}} f)(g)=\{f,g\}$.
We call $f_i$'s {\em the dependent variables}.

In Theorem 1.1 of \cite{NY0012028}, 
introducing the Poisson central {\em parameter variables} $\av_i$ ($i\in I$), 
they construct the birational Weyl group action by
\begin{align}
 &
 s_i(\av_j) = \av_j-a_{ij}\av_i, \quad
 s_i(f_i)=f_i,
 \notag 
 \\ &
 s_i(f_j) 
 = \exp(\ad_{\{,\}}\av_i\log f_i)(f_j)
 = \sum_{k=0}^{-a_{ij}} \frac{(\av_i)^k}{k!}(\ad_{\{,\}} f_i)^k(f_j) f_i^{-1}
 \quad (i\ne j).
 \label{eq:s_i(f_j)-classical}
\end{align}
These formulas shall be canonically quantized by 
\eqref{eq:s_i(beta)}, 
\eqref{eq:s_i(f_i)}, and 
\eqref{eq:s_i(f_j)-KM}, respectively.

Moreover, in Theorem 1.2 of \cite{NY0012028}, 
they introduce {\em Laurent $\tau$-monomials} $\tau^\mu$ 
for integral weights $\mu\in P$ 
and extend the birational Weyl group action to the $\tau$-monomials by
\begin{equation}
 s_i(\tau^\mu) 
 = f_i^{\bra\av_i,\lambda\ket} \tau^{s_i(\mu)}
 = f_i^{\bra\av_i,\lambda\ket} \tau^{\mu-\bra\av_i,\mu\ket\alpha_i},
 \label{eq:s_i(tau^mu)-classical1}
\end{equation}
where $\av_i$'s are identified with the simple coroots, 
$\alpha_i$'s are the simple roots, 
and $\bra\,,\,\ket$ denotes the canonical pairing between the coroot lattice $\Qv$
and the weight lattice $P$.
This action on the Laurent $\tau$-monomials shall
be quantized by \eqref{eq:s_i(tau^mu)}, the appearance of which is same as
\eqref{eq:s_i(tau^mu)-classical1}.

They deal with the $\tau$-cocycle in \cite{NY0012028}.
However, for the compatibility with the quantum case, 
we equivalently introduce {\em the $\tau$-functions} $\tau_{(w(\mu))}$ 
for $w(\mu)\in WP_+$ by
\begin{equation*}
 \tau_{(w(\mu))} = w(\tau^\mu) \quad (w\in W, \mu\in P_+).
\end{equation*}
Note that each $w(\tau^\mu)$ depends only on $w(\mu)$ 
because $s_i(\tau^\mu)=\tau^\mu$ if $\bra\av_i,\mu\ket=0$.
For any $w(\mu)\in WP_+$, there exists a unique rational function $\phi_w(\mu)$
of $f_i$'s and $\av_i$'s such that $\tau_{(w(\mu))}=\phi_w(\mu)\tau^{w(\mu)}$, 
where $\phi_w(\mu)$ is called {\em the $\tau$-cocycle} in \cite{NY0012028}.
Since the quantized version of $\phi_w(\mu)$ does not commute 
with the quantized $\tau$-monomials in general, 
we shall deal with the quantized version of $\tau_{(w(\mu))}$ in \secref{sec:QTF}.

In this paper, for the fundamental weights $\Lambda_i$, 
we call $\tau_i=\tau^{\Lambda_i}$ {\em the $\tau$-variables}.
Although they call only $\tau_i$'s the $\tau$-functions in \cite{NY0012028},
we call all $\tau_{(w(\mu))}$ ($w\in W$, $\mu\in P_+$) the $\tau$-functions. 
One should not be confused by the difference of terminologies.

The main result of \cite{NY0012028} is the regularity of $\phi_w(\mu)$ 
for any dominant integral weight $\mu\in P_+$ (Theorem 1.3 of \cite{NY0012028}).
In other words, they prove that, for any $\mu\in P_+$ and any $w\in W$, 
the $\tau$-function $\tau_{w(\mu)}$ is a polynomial in $f_i$'s and $\av_i$'s.

In \cite{Yamada9808002}, one of the author of $\cite{NY0012028}$ finds
the determinant formulas of the $\tau$-functions 
for the birational Weyl group actions of type $A^{(1)}_{n-1}$ and $A_\infty$,
and interprets them as Pl\"ucker coordinates of
the universal Grassmann manifolds in the Sato theory of soliton equations \cite{Sato-Sato}.
The determinant formulas immediately lead to 
the regularity of the $\tau$-functions of type $A$.
In \cite{NY0012028}, they generalize the Sato theoretic interpretation of 
the $A$-type $\tau$-functions to the case for any symmetrizable GCM
and show the regularity of the $\tau$-functions for any type.

The regularity of the $\tau$-functions proved by Noumi and Yamada \cite{Yamada9808002}, \cite{NY0012028}
has many corollaries which state polynomiality of certain special rational functions. 
In particular, polynomialities of rational functions, 
which give special solutions of the bilinear forms of the Painlev\'e equations 
and are generated by the B\"acklund transformations, 
are corollaries of the regularity of the $\tau$-functions
for the birational Weyl group actions.

For example, let $Q_m(x)$ ($m\in\Z_{\geqq0})$ be the rational functions 
defined by the following recurrence equation:
\begin{equation}
 Q_{m-1} Q_{m+1} = Q_m'' Q_m - (Q_m')^2 + (x^2+2m-1)Q_m^2,
 \quad Q_0=Q_1=1.
\label{eq:Qm}
\end{equation}
In \cite{OkamotoIII}, 
using the analysis of the Painlev\'e equations in \cite{Okamoto1981},
Okamoto proves that all $Q_m(x)$ are polynomials in $x$
(Proposition 5.6 of \cite{OkamotoIII}).
The polynomials $Q_m(x)$ are called the Okamoto polynomials.
The polynomiality of $Q_m(x)$ is a corollary of the regularity of
the $\tau$-functions for the birational Weyl group action of type $A^{(1)}_2$
(Theorem 4.3 of \cite{NY9708018}).
It is non-trivial to show that 
the right-hand side of the recurrence equation \eqref{eq:Qm} 
is divisible by $Q_{m-1}(x)$.
The original proof of Okamoto is not purely algebraic. 
On the other hand, 
the regularity of $\tau$-functions for the birational Weyl group action
has a purely algebraic proof.
For other examples of special polynomials for the Painlev\'e equations, 
see also \cite{Yamada-SP} and references therein.


\subsection{Quantization}

In this paper, we shall introduce the quantum $\tau$-functions 
(\secref{sec:def-tau}) and prove their regularity.  
But the method to prove the regularity is completely different
from the one in the classical case of \cite{NY0012028}.
In order to prove the regularity of the quantum $\tau$-functions, 
we shall use the translation functors in the representation theory
(\secref{sec:reg-KM}). 

The regularity in the quantized case implies the regularity in the classical case
through the classical limit.
Therefore we obtain another purely algebraic proof of the regularity 
of the classical $\tau$-functions.
In particular, 
the polynomiality of the special rational functions 
for the Painlev\'e equations generated by the B\"acklund transformations 
can be derived from the theory of the translation functors for the Kac-Moody algebras.
This could be a surprising relationship.

We summarize the implications as below:
\begin{align*}
 &
 \text{
   $\exists$ exact functor 
   $T(M)\subset M\otimes L(\mu)$ with 
   $T(M(w\circ\lambda))=M(w\circ(\lambda+\mu))$
 }
 \\ &
 \implies
 \text{regularity of the quantum $\tau$-functions}
 \\ &
 \implies
 \text{regularity of the classical $\tau$-functions}
 \\ &
 \implies
 \text{polynomiality of special rational solutions of the Painlev\'e equations}.
\end{align*}
Here we denote 
by $w\circ\lambda$ the shifted action of $w\in W$ on $\lambda\in P_+$, 
by $M(w\circ\lambda)$ the Verma module with highest weight $w\circ\lambda$, 
by $L(\mu)$ the simple quotient of $M(\mu)$ for $\mu\in P_+$,
and by $T=T_\lambda^{\lambda+\mu}$ the translation functor.

The major difficulty of quantizing the $\tau$-functions was the fact that
we did not have a natural Poisson algebra which contains the $\tau$-variables $\tau_i$
(or the $\tau$-monomials $\tau^\mu$).
Roughly speaking, ``canonical quantization'' stands for
replacement of the Poisson brackets in a Poisson algebra
with the commutators in the corresponding non-commutative associative algebra.
But we did not have Poisson brackets for the $\tau$-variables.
We should find the appropriate commutation relations for the quantum $\tau$-variables
under the situation where the Poisson brackets are unknown.

The answer is very simple. 
The $\tau$-variables $\tau_i$ ($i\in I$) are defined 
to be the exponentials of the canonical conjugate variables 
of the parameter variables $\av_i$ ($i\in I$).
More precisely, we define $\tau_i$ by $\tau_i = \exp(\d/\d\av_i)$ ($i\in I$). 
Then we have $\tau_i\tau_j=\tau_j\tau_i$, $\tau_i\av_j=(\av_j+\delta_{ij})\tau_i$ ($i,j\in I$)
and $\tau_i f_j = f_j \tau_i$.
The quantum $\tau$-variables are difference operators 
of the parameter variables.
More generally we assume that $\tau^\lambda\tau^\mu=\tau^{\lambda+\mu}$, 
$\tau^\mu\av_j=(\av_i+\bra\av_i,\mu\ket)\tau^\mu$,
and $\tau^\mu f_j=f_j\tau^\mu$
for integral weights $\mu,\lambda\in P$.

In the classical case, 
we assume that $\{\tau^\lambda,\tau^\mu\}=0$,
$\{\tau^\mu,\av_j\}=\bra\av_i,\mu\ket\tau^\mu$, and $\{\tau_i,f_j\}=0$.
Then the formula \eqref{eq:s_i(tau^mu)-classical1} can be derived as follows: 
\begin{align}
 s_i(\tau^\mu) 
 &
 = \exp(\ad_{\{,\}} \av_i \log f_i)(\tau^{s_i(\mu)})
 = \exp(\bra\av_i,\mu\ket \log f_i)\tau^{s_i(\mu)}
 = f_i^{\bra\av_i,\mu\ket}\tau^{s_i(\mu)}.
 \label{eq:s_i(tau^mu)-classical2}
\end{align}
See also \eqref{eq:s_i(f_j)-classical}.
Therefore the birational action of $s_i$ is uniformly written in the form 
$s_i(a) = \exp(\ad \av_i\log f_i)(\ts_i(a))$,
where $\ts_i$ stands for the Weyl group action on the parameter variables
(or coroots) and the integral weights which trivially acts on $f_i$'s.

In the quantum case, we shall construct the quantum birational Weyl group action by
$s_i(a)=f_i^{\av_i} \ts_i(a) f_i^{-\av_i}$, 
where $f_i$'s are the generators of 
the associative algebra the fundamental relations of which
are the Serre (or $q$-Serre) relations,
and are called the (quantum) dependent variables.
(For details, see \secref{sec:A}.)
This is an almost straightforward canonical quantization 
of the classical birational Weyl group action $s_i(a) = \exp(\ad \av_i\log f_i)(\ts_i(a))$.

The fractional powers $f_i^{\av_i}$ shall be constructed in \secref{sec:f^beta}.
The parameter variables are identified with the simple coroots also in the quantum case. 
For $\lambda\in P$ and a function $a$ of the parameter variables $\av_i$ ($i\in I$) 
which contains the fractional powers of $f_i$ ($i\in I$), 
we denote by $\phi_\lambda(a)$ the value of $a$ at $\lambda$, 
which is obtained by the substitution of $\bra\av_i,\lambda\ket$ into $\av_i$.
We shall define the fractional powers so that $a=0$ if and only if $\phi_\lambda(a)=0$
for all integral weights $\lambda\in P$.
For any rational function $a$ of $f_i$'s and $\av_i$'s,
it is sufficient for $a=0$ that $\phi_{\lambda+\rho}(a)=0$ 
for all dominant integral weights $\lambda\in P_+$,
where $\rho$ denotes the Weyl vector $\sum_{i\in I}\Lambda_i$.
Therefore the calculations of the quantum birational Weyl group action
reduces to those of 
$f_i^{\bra\av_i,\lambda+\rho\ket} \phi_{\lambda+\rho}(a) f_i^{-\bra\av_i,\lambda+\rho\ket}$ 
for dominant integral weights $\lambda\in P_+$.

For the proof of the regularity of the quantum $\tau$-functions, 
we can assume that $f_i$'s generate the lower triangular part of $U(\g)$ or $U_q(\g)$,
where $\g$ is the Kac-Moody Lie algebra of type $[a_{ij}]_{i,j\in I}$.
Let $M(\lambda)$ be the Verma module over $U(\g)$ or $U_q(\g)$ 
with highest weight $\lambda\in P$.
Denote by $v_\lambda$ a highest weight vector of $M(\lambda)$.
Then, for each $\lambda\in P$ with $\bra\av_i,\lambda\ket\geqq0$, 
the vector $f_i^{\bra\av_i,\lambda+\rho\ket}v_\lambda$ 
is a singular vector with weight $s_i\circ\lambda=s_i(\lambda+\rho)-\rho$.
In this way, we can relate the quantum birational Weyl group action 
with the singular vectors in the Verma modules with dominant integral
highest weights.
For details, see \secref{sec:sing}.

Consequently, we can reduce the regularity of the quantum 
$\tau$-functions to the divisibility (from the right) of the singular vectors
in $M(\lambda+\mu)$ by the corresponding singular vectors in $M(\lambda)$ 
for any $\lambda,\mu\in P_+$ (\propref{prop:reduce}).
The translation functor $T_\lambda^{\lambda+\mu}$ 
for the symmetrizable Kac-Moody algebra $\g$
makes a connection between the singular vectors 
in $M(\lambda)$ and those in $M(\lambda+\mu)$
and then proves the divisibility of the singular vectors
in the Kac-Moody case (\secref{sec:reg-KM}).
The equivalence (shown by Etingof and Kazhdan in \cite{EK-VI})
between the category $\O$ for $\g$ and that of $U_q(\g)$
shows that the divisibility of the singular vectors 
in the Verma modules over $U_q(\g)$
also reduces to the theory of the translation functor 
for the symmetrizable Kac-Moody algebra (\secref{sec:reg-q}).
Thus the regularity of the quantum $\tau$-functions is proved
both in the Kac-Moody case and in the $q$-difference case. 
This is the main result of this paper.


\subsection{Conventions}

We adopt the following conventions.

The term ``quantization'' stands for ``canonical quantization'',
not for ``$q$-difference deformation''.
For example, ``the quantum $q$-Hirota-Miwa equation''
does not mean ``the $q$-difference deformation of the Hirota-Miwa equation''
but ``the $q$-difference deformation of the canonically quantized Hirota-Miwa equation''. 
(See \secref{sec:QHME} for details.)

We shall deal with both canonical quantizations and their $q$-difference analogues. 
For example, for a symmetrizable Kac-Moody Lie algebra $\g$, 
the universal enveloping algebra $U(\g)$ can be regarded as 
the canonical quantization of the symmetric algebra $S(\g)$,
and the $q$-difference deformation of $U(\g)$ is denoted by $U_q(\g)$.
We shall construct the quantum $\tau$-functions both for $U(\g)$ and for $U_q(\g)$.

A field is always commutative.
A skew field stands for a possibly non-commutative field.
An associative algebra over a field shall be always with the unit $1$.
Denote the set of all non-negative integers by $\Z_{\geqq0}$.


\paragraph{Acknowledgements.}
\ACKNOWLEDGEMENTS


\section{Quantization of birational Weyl group actions}
\label{sec:QBWGA}

In this section, we extend 
the quantized birational Weyl group action
on the dependent variables $f_i$ and the parameter variables $\av_i$
(Theorem 4.3 of \cite{Kuroki2008}) to the $\tau$-variables $\tau_i$.


\subsection{Symmetrizable GCM and Weyl group}
\label{sec:GCM}

Throughout this paper, 
a matrix $[a_{ij}]_{i,j\in I}$ stands for 
a symmetrizable generalized Cartan matrix 
(GCM for short) symmetrized by positive integers $\{d_i\}_{i\in I}$.
In other words, we assume that $[a_{ij}]_{i,j\in I}$ is an integer matrix with
\begin{equation*}
 a_{ii}=2; \quad
 a_{ij}\leqq 0\ \text{if $i\ne j$}; \quad
 a_{ij}=0\ \text{if and only if}\ a_{ji}=0; \quad
 d_ia_{ij} = d_ja_{ji}.
\end{equation*}
Let $\Qv$ be a free $\Z$-module and set $P=\Hom(\Qv,\Z)$.
Denote by $\bra\ ,\ \ket:\Qv\times P\to\Z$ the canonical pairing.
Assume that $\av_i\in \Qv$ and $\alpha_i\in P$ ($i\in I$) 
satisfy the following conditions:
(1) $\bra\av_i,\alpha_j\ket=a_{ij}$ ($i,j\in I$);
(2) $\{\av_i\}_{i\in I}$ is linearly independent over $\Z$;
(3) $\alpha_i\ne\alpha_j$ if $i\ne j$.
We call $\Qv$ the coroot lattice, $P$ the weight lattice, 
$\av_i$'s the simple coroots, and $\alpha_i$'s the simple roots.
We set $P_+=\{\,\lambda\in P\mid\bra\av_i,\lambda\ket\geqq 0\ (i\in I)\,\}$,
$Q=\bigoplus_{i\in I}\Z\alpha_i$, and
$Q_+=\bigoplus_{i\in I}\Z_{\geqq 0}\alpha_i$.
We call an element of $P_+$ a dominant integral weight
and $Q$ the root lattice.
Assume that weights $\Lambda_j\in P_+$ ($j\in I$) satisfy
$\bra\av_i,\Lambda_j\ket=\delta_{ij}$ ($i,j\in I$).
We call $\Lambda_j$'s the fundamental weights.
The Weyl vector $\rho$ is defined by $\rho=\sum_{i\in I}\Lambda_i$.

Let $W$ be the Weyl group of the GCM $[a_{ij}]_{i,j\in I}$,
namely $W$ is defined to be the group generated by $\{s_i\}_{i\in I}$
with following fundamental relations:
$s_is_j=s_js_i$         if $(a_{ij},a_{ji})=(0,0)$;
$s_is_js_i=s_js_is_j$   if $(a_{ij},a_{ji})=(-1,-1)$;
$(s_is_j)^2=(s_js_i)^2$ if $(a_{ij},a_{ji})=(-1,-2)$;
$(s_is_j)^3=(s_js_i)^3$ if $(a_{ij},a_{ji})=(-1,-3)$;
$s_i^2=1$.
The Weyl group $W$ acts on $\Qv$ and $P$ by
\begin{equation*}
 s_i(\beta)=\beta-\bra\beta,\alpha_i\ket\av_i, \quad
 s_i(\lambda)=\lambda-\bra\av_i,\lambda\ket\alpha_i.
\end{equation*}
In particular, we have
$s_i(\av_j)=\av_j-a_{ji}\av_i$ and
$s_i(\Lambda_j)=\Lambda_j-\delta_{ij}\alpha_i$.
The Weyl group action preserves 
the canonical pairing $\bra\,,\,\ket$ between $\Qv$ and $P$:
\begin{equation*}
 \bra w(\beta),w(\lambda)\ket 
 = \bra\beta,\lambda\ket
 \quad
 \text{for $\beta\in\Qv$, $\lambda\in P$, $w\in W$}.
\end{equation*}


\subsection{Quantum algebras of dependent variables}
\label{sec:A}

Let $\h$ be the vector space over $\C$ isomorphic to $\Qv_\C=\Qv\otimes\C$.
Denote by $h_i$ (resp.\ $h_\beta$) 
the element of $\h$ corresponding to $\av_i\otimes1\in\Qv_\C$
(resp.\ $\beta\in\Qv_\C$).
Note that we shall not identify $h_\beta$ with $\beta\in\Qv_\C$.

Let $\g$ be the Kac-Moody Lie algebra of type $[a_{ij}]_{i,j\in I}$ 
over $\C$, namely $\g$ is defined to be the Lie algebra over $\C$ 
generated by $\{\,e_i,f_i,h\mid i\in I, h\in\h \,\}$ with fundamental relations:
\begin{align*}
 &
 [h_\beta,h_\gamma] = 0, \quad h_\beta+h_\gamma=h_{\beta+\gamma} 
 \quad (\beta,\gamma\in\Qv_\C), 
 \\ &
 [h_\beta, e_j] =  \bra\beta,\alpha_j\ket e_j, \quad
 [h_\beta, f_j] = -\bra\beta,\alpha_j\ket f_j \quad (\beta\in\Qv_\C), \quad
 [e_i,f_j] = \delta_{ij}h_i,
 \\ &
 \sum_{k=0}^{1-a_{ij}} 
 (-1)^k e_i^{(1-a_{ij}-k)}e_je_i^{(k)} = 0,
 \quad
 \sum_{k=0}^{1-a_{ij}} 
 (-1)^k f_i^{(1-a_{ij}-k)}f_jf_i^{(k)} = 0
 \quad\text{if $i\ne j$},
\end{align*}
where $e_i^{(k)}=e_i/k!$ and $f_i^{(k)}=f_i/k!$.
The last two relations are called the Serre relations.
Let $\n_-$ (resp.\ $\n_+$) be 
the Lie subalgebra of $\g$ generated by $\{f_i\}_{i\in I}$
(resp.\ by $\{e_i\}_{i\in I}$).
Denote by $U(\mathfrak{a})$ the universal enveloping algebra of a Lie 
algebra $\mathfrak{a}$.

Note that $\Qv$ is not regarded as a subset of $\h$.
We shall assume that elements of $\Qv$ commute with $f_i$'s.

In order to deal with $q$-difference analogues, 
we introduce the $q$-numbers, $q$-factorials, and 
the $q$-binomial coefficients by
\begin{align*}
 &
 [a]_q = \frac{q^a-q^{-a}}{q-q^{-1}}, \qquad
 [k]_q! = [1]_q[2]_q\cdots[k]_q, \quad
 \\ &
 \qbinom{a}{k}_q =
 \frac{[a]_q[a-1]_q[a-2]_q\cdots[a-k+1]_q}{[k]_q!}
 \quad  (k\in\Z_{\geqq0}).
\end{align*}
Put $q_i = q^{d_i}$ for $i\in I$.

Let $U_q(\g)$ be the $q$-difference analogue of $U(\g)$, 
namely $U_q(\g)$ is defined to be the associative algebra over $\C(q)$ 
generated by $\{\,e_i,f_i,q^\beta\mid i\in I, \beta\in\Qv \,\}$ 
with fundamental relations:
\begin{align*}
 &
 q^\beta q^\gamma = q^{\beta+\gamma}, \quad q^0=1
 \quad (\beta,\gamma,0\in\Qv),
 \\ &
 q^\beta e_j q^{-\beta} = q^{ \bra\beta,\alpha_j\ket} e_j, \quad
 q^\beta f_j q^{-\beta} = q^{-\bra\beta,\alpha_j\ket} f_j
 \quad (\beta\in\Qv), \quad
 [e_i,f_j] = \delta_{ij}[h_i]_{q_i},
 \\ &
 \sum_{k=0}^{1-a_{ij}} 
 (-1)^k e_i^{(1-a_{ij}-k)}e_je_i^{(k)} = 0,
 \quad
 \sum_{k=0}^{1-a_{ij}} 
 (-1)^k f_i^{(1-a_{ij}-k)}f_jf_i^{(k)} = 0
 \quad\text{if $i\ne j$},
\end{align*}
where $e_i^{(k)}=e_i/[k]_{q_i}!$ and $f_i^{(k)}=f_i/[k]_{q_i}!$.
The last two relations are called the $q$-Serre relations.
Let $U_q(\n_-)$ (resp.\ $U_q(\n_+)$) be 
the subalgebra of $U_q(\g)$ generated by $\{f_i\}_{i\in I}$
(resp.\ $\{e_i\}_{i\in I}$).

The case where we deal with the Kac-Moody Lie algebra $\g$
(resp.\ the $q$-difference deformation $U_q(\g)$ of $U(\g)$)
shall be called {\em the Kac-Moody case} 
(resp.\ {\em the $q$-difference case}).

In the Kac-Moody (resp.\ $q$-difference) case, 
let $A$ be a residue class algebra of $U(\n_-)$ (resp.\ $U_q(\n_-)$)
and assume that $A$ is an integral domain, namely
an associative algebra without non-zero zero divisors.
Denote the images of $f_i$'s in $A$ by the same symbols.
We also assume that $f_i\ne 0$ in $A$ for all $i\in I$.
We call $f_i$'s {\em the dependent variables}.
We shall construct the quantum birational Weyl group action on 
certain extensions of $A$.

We denote by $Q(R)$ the skew field of fractions of an Ore domain $R$.
Any element of $Q(R)$ can be expressed as $as^{-1}$ and $t^{-1}b$
for some $a,b,s,t\in A$ with $s\ne 0$ and $t\ne 0$.
We denote by $F(x_1,\ldots,x_n)$ the the skew field of fractions 
of the polynomial ring $F[x_1,\ldots,x_n]$ over a skew field $F$.

If the GCM is of finite or affine type, 
then $A$ is always an Ore domain (Theorem 2.12 of \cite{Kuroki2008}).
When $A$ is an Ore domain, we denote $Q(A)$ by $K$.

\begin{remark}
\label{remark:Ore}
 We refer \cite{S-1971}, \cite{Jat-1986}, \cite{MR-2001}, \cite{GW-2004},
 Section 3.6 of \cite{Dixmier}, 
 and Section 2 of \cite{Kuroki2008}
 for the theory of localization of non-commutative rings.

 A Noetherian domain is always an Ore domain 
 (2.1.15 of \cite{MR-2001}, Corollary 6.7 of \cite{GW-2004}).
 The universal enveloping algebra of a finite dimensional Lie algebra
 is a Noetherian domain and hence an Ore domain.
 The $q$-difference deformation $U_q(\g)$ of $U(\g)$ 
 and its lower triangular part $U_q(\n_-)$ are always integral domains 
 (Section 7.3 of \cite{Jos-1995}).
 If the GCM is of finite type, then $U_q(\g)$ and $U_q(\n_-)$ are
 Noetherian (Section 7.4 of \cite{Jos-1995}) and hence Ore domains.

 Let $R$ be an associative algebra over a field with
 a increasing filtration $R=\bigcup_{k=0}^\infty R_k$ such that
 each $R_k$ is a finite dimensional subspace of $R$ and $1\in R_0$.
 We call $R$ {\em the tempered domain} if $R$ is an integral domain and 
 the convergence radius of $\sum_{k=0}^\infty \dim(R_k) z^k$ 
 is not less than $1$.
 A tempered domain is always an Ore domain (Lemma 2.9 of \cite{Kuroki2008},
 Lemma 1.2 of \cite{RCW}).
 This result is very useful for proving that a given algebra is an Ore domain. 
 In particular, it follows that, if the GCM is of finite or affine type, 
 then $U(\g)$, $U_q(\g)$, and their all subquotient domains 
 are Ore domains (Theorem 2.12 of \cite{Kuroki2008}).

 Consider  the polynomial ring $R=F[x_1,\ldots,x_n]$ over a skew field $F$
 and assume that $m=0,1,2,\ldots,n$ and $c_1,\ldots,c_m$ are central
 elements of $F$.
 Then, applying Theorem 2.1 of \cite{S-1971} to $R$, we obtain that 
 $S=\{\, f\in R\mid f(c_1,\ldots,c_m,x_{m+1},\ldots,x_n)\ne 0\,\}$
 is an Ore set in $R$.
 \qed
\end{remark}

\begin{example}
 Assume that $\eps_{ij}\in\{0,\pm1\}$, $\eps_{ji}=-\eps_{ij}$, and 
 $\eps_{ij}\ne 0$ if and only if $a_{ij}\ne 0$ and $i\ne j$.
 Set $c_{ij}=\eps_{ij}d_ia_{ij}$ for $i,j\in I$.
 If $A$ is the associative algebra over $\C$ generated by
 $\{f_i\}_{i\in I}$ with fundamental relations
 $[f_i,f_j] = c_{ij}$ $(i,j\in I)$,
 then $A$ is a quotient Ore domain of $U(\n_-)$.
 If $A$ is the associative algebra over $\C(q)$ generated by
 $\{f_i\}_{i\in I}$ with fundamental relations
 $f_j f_i = q^{c_{ij}}f_if_j$ $(i,j\in I)$, 
 then $A$ is a quotient Ore domain of $U_q(\n_-)$.
 In the both cases, we have the skew field $K=Q(A)$ of fractions of $A$.
 \qed
\end{example}

If $R$ is an associative algebra and $S$ is an Ore set in $R$, 
then we have the localization $R[S^{-1}]$ of $R$ with respect to $S$.
For any $c\in R[S^{-1}]$, there exist some $a,b\in R$ and $s,t\in S$
such that $c=as^{-1}=t^{-1}b$.

The multiplicative subset of $A$ generated by $\{f_i\}_{i\in I}$
is an Ore set in $A$ owing to the Serre and $q$-Serre relations.
Therefore we have the localization $\tA=A[f_i^{-1}|i\in I]$ 
of $A$ with respect to it.

We have $A\subset\tA$ in any case
and $\tA\subset K=Q(A)$ in the case where $A$ is an Ore domain.

\begin{example}
\label{example:C[x,d]A2}
 Assume that the GCM is of type $A_2$: $I=\{1,2\}$, $a_{ii}=2$, 
 $a_{12}=a_{21}=-1$, $d_i=1$. Then the Weyl algebra $A=\C[x,\d]$, 
 where $\d=d/dx$, can be regarded as a quotient Ore domain of $U(\n_-)$.
 The surjective algebra homomorphism from $U(\n_-)$ onto $A=\C[x,\d]$ 
 is given by $f_1\mapsto x$ and $f_2\mapsto\d$.
 Then the multiplicative subset of $A=\C[x,\d]$ generated by $x$ and $\d$
 is an Ore set in $A=\C[x,\d]$ and we obtain $\tA=\C[x^{\pm1},\d^{\pm1}]$.
 For any polynomial $f\in\C[x]$, we have
 \begin{equation*}
  \d^{-1}f 
  = \sum_{k=0}^\infty (-1)^k f^{(k)}\d^{-k-1}
  = f\d^{-1} - f'\d^{-2} + f''\d^{-3} - \cdots
  \quad \text{in $\tA=\C[x^{\pm1},\d^{\pm1}]$}.
 \end{equation*}
 The right-hand side reduces to a finite sum.
 \qed
\end{example}

\begin{example}
\label{example:C[x,d]}
The Weyl algebra $A=\C[x,\d]$, $\d=d/dx$, can be regarded as 
a quotient Ore domains of $U(\n_-)$ of various affine types:
\begin{align*}
& D^{(1)}_4: \quad f_2=\d, \quad f_i=x-a_i \quad (a_i\in\C, i=0,1,3,4), \\
& B^{(1)}_3: \quad f_1=\d, \quad f_2=x, \quad f_3=x-a, \quad f_0=(x-b)^2 \quad (a,b\in\C), \\
& A^{(1)}_3: \quad f_1=\d, \quad f_2=x, \quad f_3=x-a, \quad f_0=\d-b \quad (a,b\in\C), \\
& G^{(1)}_2: \quad f_1=\d, \quad f_2=x, \quad f_0=(x-a)^3 \quad (a\in\C), \\
& A^{(1)}_2: \quad f_1=\d, \quad f_2=x, \quad f_0=\d+x, \\
& D^{(2)}_5: \quad f_1=\d, \quad f_2=x^2, \quad f_0=(x-a)^2 \quad (a\in\C), \\
& C^{(1)}_2: \quad f_1=\d, \quad f_2=x^2, \quad f_0=\d-a    \quad (a\in\C), \\
& A^{(2)}_2: \quad f_1=\d, \quad f_0=x^4, \\
& A^{(1)}_1: \quad f_1=\d, \quad f_0=\d+x^2.
\end{align*}
The GCM's of type $D^{(1)}_4$, $A^{(1)}_3$, and $A^{(1)}_2$ are simply-laced, 
those of $C^{(1)}_2$ and $A^{(1)}_1$ are symmetric, and
those of $B^{(1)}_3$, $G^{(1)}_2$, $D^{(2)}_5$, and $A^{(2)}_2$ are not symmetric 
but foldings of $D^{(1)}_4$.

For example, in the case of $G^{(1)}_2$, 
from $f_1=\d$, $f_2=x$, and $f_0=(x-a)^3$,
it follows that $[f_1,[f_1,f_2]]=0$, $[f_2,[f_2,f_1]]=0$, 
$[f_1,[f_1,[f_1,[f_1,f_0]]]]=0$, $[f_0,[f_0,f_1]]=0$, and $[f_2,f_0]=0$.
These are the Serre relations of type $G^{(1)}_2$.

Recall that the Painlev\'e equations $\PVI$, $\PV$, $\PIV$, $\PIII$, and $\PII$
have the birational Weyl group symmetries of type 
$D^{(1)}_4$, $A^{(1)}_3$, $A^{(1)}_2$, $C^{(1)}_2$, and $A^{(1)}_1$, respectively.
By Proposition 5.13 of \cite{Yamakawa}, 
we have the following isomorphisms of the affine Weyl groups:
\begin{alignat*}{2}
 & W(B^{(1)}_3) \cong W(A^{(1)}_3)\rtimes\Z/2\Z, \quad &
 & W(G^{(1)}_2) \cong W(A^{(1)}_2)\rtimes\Z/2\Z, \\
 & W(D^{(2)}_5) \cong W(C^{(2)}_2)\rtimes\Z/2\Z, \quad &
 & W(A^{(2)}_2) \cong W(A^{(1)}_1)\rtimes\Z/2\Z.
\end{alignat*}
In fact, the above list is related to the canonical quantization of the Painlev\'e equations.
For the construction of the quantum Painlev\'e equations, 
see \cite{JNS}, \cite{NGR}, \cite{Nagoya2011}, and \cite{Nagoya2012}.
The Weyl group symmetry of the quantum Painlev\'e equations are naturally interpreted
from the perspective of the quantum birational Weyl group actions 
defined in this paper or in \cite{Kuroki2008}.
\qed
\end{example}


\subsection{Quantum algebras with parameter variables}
\label{sec:A^pa}

In the Kac-Moody case, we consider
the polynomial rings 
$A^\pa=A[\beta|\beta\in\Qv]$ and $\tA^\pa=\tA[\beta|\beta\in\Qv]$
generated by any free basis of the coroot lattice $\Qv$ over $A$ and $\tA$ respectively,
where $\beta\in\Qv$'s are central in them:
\begin{equation*}
 \beta f_j = f_j \beta, \quad
 \beta\gamma = \gamma\beta \quad
 (\beta,\gamma\in\Qv).
\end{equation*}
We naturally regard the coroot lattice $\Qv$ 
as a subset of $A^\pa$ and $\tA^\pa$.
We call $\beta\in\Qv$ {\em the parameter variables}.

In the $q$-difference case, we consider
the Laurent polynomial rings 
$A^\pa=A[q^\beta|\beta\in\Qv]$ and
$\tA^\pa=\tA[q^\beta|\beta\in\Qv]$
spanned by $\{q^\beta\}_{\beta\in\Qv}$ 
over $A$ and $\tA$ respectively,
where $q^\beta$'s are central in them:
\begin{equation*}
 q^\beta f_j = f_j q^\beta, \quad
 q^\beta q^\gamma = q^{\beta+\gamma} \quad
 (\gamma,\beta\in\Qv).
\end{equation*}
We call $q^\beta$'s {\em the parameter variables}.

In the both cases, if $A$ is an Ore domain, 
then $A^\pa$ is also an Ore domain.
We denote $Q(A^\pa)$ by $K^\pa$ shortly, 
namely, $K^\pa=K(\beta|\beta\in\Qv)$ in the Kac-Moody case
and $K^\pa=K(q^\beta|\beta\in\Qv)$ in the $q$-difference case.
We have $K^\pa=Q(\tA^\pa)$.

We call $A^\pa$, $\tA^\pa$, and $K^\pa$ 
{\em the quantum algebras with parameter variables}.
Also in the following, 
the symbol ``$(\ )^\pa$'' stands for {\em ``with parameter variables''}.

In the Kac-Moody (resp.\ $q$-difference) case,
for each $\lambda\in P$, we define the algebra homomorphism 
$\phi_\lambda:\tA^\pa\to\tA$ by
\begin{equation*}
 \phi_\lambda(a)=a, \quad
 \phi_\lambda(\beta) = \bra\beta,\lambda\ket \quad
 (\text{resp.}\ \phi_\lambda(q^\beta) = q^{\bra\beta,\lambda\ket}) \quad
 (a\in\tA, \beta\in\Qv).
\end{equation*}
This homomorphism $\phi_\lambda$ 
substitutes $\bra\beta,\lambda\ket$ into $\beta\in\Qv$.
Thus we obtain the injective algebra homomorphism 
$\phi:\tA^\pa\to\tA^P$ by $\phi(a)=(\phi_\lambda(a))_{\lambda\in P}$
($a\in\tA^\pa$) and identify $\tA^\pa$ with its image in $\tA^P$.
Since $\phi(A^\pa)\subset A^P$, we obtain the injective algebra
homomorphism $\phi:A^\pa\to A^P$
and identify $A^\pa$ with its image in $A^P$.

In this paragraph, we assume that $A$ is an Ore domain.
For each $\lambda\in P$, 
define the multiplicative subset $S_\lambda$ of $A^\pa$ 
by $S_\lambda=\{\,a\in A^\pa\mid\phi_\lambda(a)\ne0\,\}$.
Then $S_\lambda$ is an Ore set in $A^\pa$.
(See the last paragraph of \remarkref{remark:Ore}.)
Thus we obtain the subalgebra $A^\pa[S_\lambda^{-1}]$ of $K^\pa$
and $\phi_\lambda:A^\pa\to A$ is uniquely extended to 
the algebra homomorphism $\phi_\lambda:\tA^\pa[S_\lambda^{-1}]\to K$.
Any element of $\tA^\pa[S_\lambda^{-1}]$ can be expressed as 
$as^{-1}$ by some $a,s\in A^\pa$ with $\phi_\lambda(s)\ne 0$
and $\phi_\lambda(as^{-1})=\phi_\lambda(a)\phi_\lambda(s)^{-1}$.
Let $A^\pa_{(P)}$ be the intersection of $A^\pa[S_\lambda^{-1}]$
for all $\lambda\in P$. 
Then $\phi_\lambda:A^\pa\to A$ is uniquely extended to
the algebra homomorphism $\phi_\lambda: A^\pa_{(P)}\to K$.
Thus we obtain the injective algebra homomorphism 
$\phi:A^\pa_{(P)}\to K^P$ by $\phi(a)=(\phi_\lambda(a))_{\lambda\in P}$
($a\in A^\pa_{(P)}$)
and identify $A^\pa_{(P)}$ with its image in $K^P$.

We have $A^\pa\subset\tA^\pa\subset\tA^P$ in any case
and $\tA^\pa\subset A^\pa_{(P)}\subset K^P$
in the case where $A$ is an Ore domain.


\subsection{Quantum difference operator algebras and $\tau$-variables}
\label{sec:D(A^pa)}

In this subsection, we shall introduce difference operators acting on 
the parameter variables and call them $\tau$-variables.

In the Kac-Moody case, 
for each $\mu\in P$, let $\tau^\mu$ be the difference operator 
acting on quantum algebras with parameter variables given by
\begin{equation*}
 \tau^\mu(f_i) = f_i, \quad
 \tau^\mu(\beta) = \beta+\bra\beta,\mu\ket \quad
 (i\in I, \beta\in\Qv).
\end{equation*}
Thus we obtain the difference operator algebras
$D(A^\pa)=A^\pa[\tau^\mu|\mu\in P]$, 
$D(\tA^\pa)=\tA^\pa[\tau^\mu|\mu\in P]$, 
$D(K^\pa)=K^\pa[\tau^\mu|\mu\in P]$.
In these algebras, we have
\begin{equation*}
 \tau^\lambda\tau^\mu=\tau^{\lambda+\mu}, \quad
 \tau^\mu f_i = f_i\tau^\mu, \quad
 \tau^\mu \beta = (\beta+\bra\beta,\mu\ket)\tau^\mu \quad
 (\lambda,\mu\in P, \beta\in\Qv).
\end{equation*}

In the $q$-difference case, 
for each $\mu\in P$, let $\tau^\mu$ be the $q$-difference operator 
acting on quantum algebras with parameter variables given by
\begin{equation*}
 \tau^\mu(f_i) = f_i, \quad
 \tau^\mu(q^\beta) = q^{\beta+\bra\beta,\mu\ket} \quad
 (i\in I, \beta\in\Qv).
\end{equation*}
Thus we obtain the $q$-difference operator algebras 
$D(A^\pa)$, $D(\tA^\pa)$, $D(K^\pa)$, 
similarly as in the Kac-Moody case.
In these algebras we have
\begin{equation*}
 \tau^\lambda\tau^\mu=\tau^{\lambda+\mu}, \quad
 \tau^\mu f_i = f_i\tau^\mu, \quad
 \tau^\mu q^\beta = q^{\beta+\bra\beta,\mu\ket}\tau^\mu \quad
 (\lambda,\mu\in P, \beta\in\Qv).
\end{equation*}

Note that $D(K^\pa)$ is defined only in the case where $A$ is an Ore domain.

For any algebra $R$ and any $\mu\in P$, 
the algebra automorphism $\tau^\mu$ of $R^P$ is 
defined by the translation \(
 \tau^\mu((a_\lambda)_{\lambda\in P}) =
 (a_{\lambda+\mu})_{\lambda\in P}
\) ($(a_\lambda)_{\lambda\in P}\in R^P$).
Thus we obtain the extended algebra $D(R^P)=R^P[\tau^\mu|\mu\in P]$.

The injective algebra homomorphism $\phi:\tA^\pa\to\tA^P$
commutes with $\tau^\mu$ for any $\mu\in P$.
Therefore $\phi$ is naturally extended to
the injective algebra homomorphism $\phi:D(\tA^\pa)\to D(\tA^P)$.
We identify $D(\tA^\pa)$ with its image in $D(\tA^P)$.

Similarly, when $A$ is an Ore domain, $\phi:A^\pa_{(P)}\to K$ is
naturally extended to
the injective algebra homomorphism $\phi:D(A^\pa_{(P)})\to D(K^P)$.
We identify $D(\tA^\pa_{(P)})$ with its image in $D(K^P)$.

We have $D(A^\pa)\subset D(\tA^\pa)\subset D(\tA^P)$ in any case
and $D(\tA^\pa)\subset D(A^\pa_{(P)})\subset D(K^P)$ 
in the case where $A$ is an Ore domain.
These algebras are called {\em the quantum difference operator algebras}.

We call $\tau^\mu$'s {\em the quantum Laurent $\tau$-monomials}.
We define {\em the quantum $\tau$-variables} $\tau_i$ ($i\in I$) 
by $\tau_i=\tau^{\Lambda_i}$.


\subsection{Quantum algebras with fractional powers}
\label{sec:f^beta}

For each $\beta\in\Qv$, we define 
{\em the fractional power} $f_i^\beta$ by \(
 f_i^\beta=(f_i^{\bra\beta,\lambda\ket})_{\lambda\in P} \in\tA^P
\). Let $\A$ be the subalgebra of $\tA^P$ generated 
by $\tA^\pa=\phi(\tA^\pa)$ and the fractional powers $f_i^\beta$ ($\beta\in\Qv$),
and $D(\A)$ the subalgebra of $D(\tA^P)$ generated
by $D(\tA^\pa)=\phi(D(\tA^\pa))$ and the fractional powers $f_i^\beta$ ($\beta\in\Qv$):
\begin{equation*}
 \A    = \tA^\pa[f_i^\beta|\beta\in\Qv], \quad
 D(\A) = \A[\tau^\mu|\mu\in P].
\end{equation*}
These algebras are called {\em the quantum algebra with fractional powers}
and {\em the quantum difference operator algebra with fractional powers}, 
respectively. In $D(\A)$, we have
\begin{equation*}
 \tau^\mu f_i^\beta = f_i^{\beta+\bra\beta,\mu\ket}\tau^\mu
 \quad (\mu\in P, \beta\in\Qv, i\in I).
\end{equation*}
For $\lambda\in P$, let $\phi_\lambda:\A\to\tA$ be the restriction on $\A$ of 
the canonical projection from $\tA^P$ onto its $\lambda$-factor.
Then we have $\phi_\lambda(f_i^\beta)=f_i^{\bra\beta,\lambda\ket}$.

Similarly, when $A$ is an Ore domain, we can construct the algebras
\begin{equation*}
 \A_{(P)}    = A^\pa_{(P)}[f_i^\beta|\beta\in\Qv], \quad
 D(\A_{(P)}) = \A_{(P)}[\tau^\mu|\mu\in P]
\end{equation*}
as subalgebras of $K^P$ and $D(K^P)$, respectively.
For $\lambda\in P$, let $\phi_\lambda:\A_{(P)}\to K$ be 
the restriction on $\A_{(P)}$ of 
the canonical projection from $K^P$ onto its $\lambda$-factor.


\subsection{Tilde action of the Weyl group}
\label{sec:tilde}

For any algebra $R$ and any $w\in W$, 
the algebra automorphism $\tw$ of $D(R^P)$ is given by
\begin{equation*}
 \tw((a_\lambda)_{\lambda\in P}) = (a_{w^{-1}(\lambda)})_{\lambda\in P},
 \quad
 \tw(\tau^\mu) = \tau^{w(\mu)}
 \quad 
 ((a_\lambda)_{\lambda\in P}\in R^P, \mu\in P).
\end{equation*}
This is called {\em the tilde action} of the Weyl group.

In the case of $R=\tA$, 
the tilde action of $w\in W$ on $D(\tA^P)$ preserves $D(\A)$
and its action on $D(\A)$ is characterized by
\begin{align*}
 &
 \tw(f_i^{\pm1})=f_i^{\pm1} \quad (i\in I),
 \\ &
 \tw(\beta)=w(\beta) \quad (\beta\in\Qv)
 \quad\text{in the Kac-Moody case},
 \\ &
 \tw(q^\beta)=q^{w(\beta)} \quad (\beta\in\Qv)
 \quad\text{in the $q$-difference case},
 \\&
 \tw(f_i^\beta)=f_i^{w(\beta)} \quad (i\in I, \beta\in\Qv),
 \\&
 \tw(\tau^\mu)=\tau^{w(\mu)} \quad (\mu\in P).
\end{align*}
Similarly, when $A$ is an Ore domain, the tilde action on $D(K^P)$
preserves $D(\A_{(P)})$.


\subsection{Quantum birational Weyl group action}
\label{sec:birat}

We are ready to construct the quantization of 
the birational Weyl group action arising from a nilpotent
Poisson algebra proposed by Noumi and Yamada \cite{NY0012028}.

\begin{lemma}[Verma identities]
\label{lemma:Verma}
 In $D(\A)$, for any $\beta,\gamma\in\Qv$, we have
 \begin{itemize}

  \item If $(a_{ij},a_{ji})=(0,0)$, then \(
   f_i^{\beta} f_j^{\gamma} = f_j^{\gamma} f_i^{\beta}
  \).

  \item If $(a_{ij},a_{ji})=(-1,-1)$, then \(
   f_i^{\beta} f_j^{\beta+\gamma} f_i^{\gamma} =
   f_j^{\gamma} f_i^{\beta+\gamma} f_j^{\beta}
  \).

  \item If $(a_{ij},a_{ji})=(-1,-2)$, then \(
   f_i^{\beta} f_j^{2\beta+\gamma} f_i^{\beta+\gamma} f_j^{\gamma} =
   f_j^{\gamma} f_i^{\beta+\gamma} f_j^{2\beta+\gamma} f_i^{\beta}
  \).

  \item If $(a_{ij},a_{ji})=(-1,-3)$, then \\
  \hphantom{\qquad} 
  \(
   f_i^{\beta} f_j^{3\beta+\gamma} f_i^{2\beta+\gamma} 
   f_j^{3\beta+2\gamma} f_i^{\beta+\gamma} f_j^{\gamma} =
   f_j^{\gamma} f_i^{\beta+\gamma} f_j^{3\beta+2\gamma} 
   f_i^{2\beta+\gamma} f_j^{3\beta+\gamma} f_i^{\beta}
  \).
 \end{itemize}
\end{lemma}

\begin{proof}
 It is sufficient to show 
 that the identities obtained by substituting $(m,n)$ to $(\beta,\gamma)$ 
 in the above identities hold for all integers $m,n$.
 But it reduces to the cases where $m,n$ are non-negative.
 The proof of the non-negative cases is found 
 in Proposition 39.3.7 of \cite{Lusztig}.
 \qed
\end{proof}

\begin{remark}
 The Verma identities (Proposition 39.3.7 of \cite{Lusztig}) 
 can be regarded as a corollary of the uniqueness (up to scalar multiples) 
 of homomorphisms between Verma modules.
 See 4.4.16 of \cite{Jos-1995}.
 \qed
\end{remark}

\begin{example}
\label{example:x-d-Verma}
 Under the setting of \exampleref{example:C[x,d]A2}
 and $\Qv=\Z\av_1\oplus\Z\av_2$, 
 the following formula holds 
 in $\A^\pa=\C[x,\d,\av_i,x^{\pm\av_i},\d^{\pm\av_i}|i=1,2]$:
 \begin{equation}
  x^a\d^{a+b}x^b = \d^b x^{a+b} \d^a,
  \label{eq:x-d-Verma}
 \end{equation}
 where $a=\beta$, $b=\gamma$, and $\beta,\gamma\in\Qv$.
 Although this is a special case of \lemmaref{lemma:Verma}, 
 we shall show the direct proof to help 
 understanding the proof of \lemmaref{lemma:Verma}.
 It is sufficient to prove that the formula 
 \eqref{eq:x-d-Verma} holds for any $a,b\in\Z$.
 Using the Leibnitz formula
 \begin{equation*}
  \d^n f = \sum_{k=0}^n\binom{n}{k}f^{(k)}\d^{n-k}
  \quad \text{in $\C[x,\d]$}
  \quad (f\in\C[x], n\in\Z_{\geqq0}),
 \end{equation*}
 we can show that for any $a,b\in\Z_{\geqq0}$, the both-hand sides
 of \eqref{eq:x-d-Verma} are equal to
 \[
  \sum_{k=0}^b k!\binom{a+b}{k}\binom{b}{k}x^{a+b-k}\d^{a+b-k}.
 \]
 The formula \eqref{eq:x-d-Verma} for $a,b\in\Z_{\geqq0}$ has been proved. 
 Replacing $(a,b)$ with $(a+b,-b)$,
 we find that the formula \eqref{eq:x-d-Verma} for $a,b\in\Z_{\geqq0}$ is
 equivalent to the one for $a\in\Z_{\geqq0}$ and $b\in\Z_{\leqq0}$
 with $a+b\geqq0$.
 Replacing $(a,b)$ with $(-a-b,a)$,
 we find that the formula \eqref{eq:x-d-Verma} for $a,b\in\Z_{\geqq0}$ is
 equivalent to the one for $a\in\Z_{\geqq0}$ and $b\in\Z_{\leqq0}$
 with $a+b\leqq0$.
 We can similarly obtain the other cases and prove 
 the formula \eqref{eq:x-d-Verma} for any $a,b\in \Z$.

 The fractional power $\d^\beta$ of $\d=d/dx$ 
 is related to the fractional calculus and the Kats middle convolution.
 See \cite{Oshima}, \cite{Nagoya2012}, and references therein.
 \qed
\end{example}

\begin{remark}
\label{remark:geometric-crystal}
 For $c=q^\beta$ ($\beta\in\Qv$), define $\be_i^c$ by $\be_i^c(a)=f_i^\beta a f_i^{-\beta}$.
 Then \lemmaref{lemma:Verma} immediately leads to the Verma identities of $\be_i^c$ ($i\in i$):
 for $c=q^\beta$ and $d=q^\gamma$, 
 \begin{itemize}

  \item If $(a_{ij},a_{ji})=(0,0)$, then \(
   \be_i^c \be_j^d = \be_j^d \be_i^c
  \).

  \item If $(a_{ij},a_{ji})=(-1,-1)$, then \(
   \be_i^c \be_j^{cd} \be_i^d =
   \be_j^c \be_i^{cd} \be_j^c
  \).

  \item If $(a_{ij},a_{ji})=(-1,-2)$, then \(
   \be_i^c \be_j^{c^2d} \be_i^{cd} \be_j^d =
   \be_j^d \be_i^{cd} \be_j^{c^2d} \be_i^c
  \).

  \item If $(a_{ij},a_{ji})=(-1,-3)$, then 
  \(
   \be_i^c \be_j^{c^3d} \be_i^{c^2d^2} \be_j^{c^3d^2} \be_i^{cd} \be_j^d =
   \be_j^c \be_i^{cd} \be_j^{c^3d^2} \be_i^{c^2d} \be_j^{c^2d} \be_i^c
  \).
 \end{itemize}
 These identities mean that $\{\be_i^c\}_{i\in I}$ can be regarded as
 a quantum version of geometric crystal 
 defined by Berenstein and Kazhdan \cite{BK2000}.
 \qed
\end{remark}

For any associative algebra $R$ and an invertible element $a\in R^\times$,
{\em the inner algebra automorphism} $\Ad(a)$ of $R$ is defined by
\begin{equation*}
 \Ad(a)(x) = a x a^{-1} \quad \text{for $x\in R$}.
\end{equation*}
Then the multiplicative group $R^\times$ acts on $R$ via $\Ad$.

\begin{definition}
 For each $i\in I$, the algebra automorphisms $\bs_i$ of $D(\A)$ 
 (and of $D(\A_{(P)})$ when $A$ is an Ore domain) 
 by $\bs_i = \Ad(f_i^{\av_i})\ts_i$. 
 \qed
\end{definition}

\begin{lemma}
\label{lemma:bs}
 The set $\{\bs_i\}_{i\in I}$ of algebra automorphisms of $D(\A)$
 (and of $D(\A_{(P)})$ when $A$ is an Ore domain) satisfies 
 the fundamental relations of the Weyl group:
 $\bs_i\bs_j=\bs_j\bs_i$         if $(a_{ij},a_{ji})=(0,0)$;
 $\bs_i\bs_j\bs_i=\bs_j\bs_i\bs_j$   if $(a_{ij},a_{ji})=(-1,-1)$;
 $(\bs_i\bs_j)^2=(\bs_j\bs_i)^2$ if $(a_{ij},a_{ji})=(-1,-2)$;
 $(\bs_i\bs_j)^3=(\bs_j\bs_i)^3$ if $(a_{ij},a_{ji})=(-1,-3)$;
 $\bs_i^2=1$.
\end{lemma}

\begin{proof}
 The relation 
 $\bs_i^2=1$ follows from $\ts_i(f_i^{\av_i})=f_i^{-\av_i}$:
 \begin{align*}
  \bs_i^2 
  = \Ad(f_i^{\av_i})\ts_i\Ad(f_i^{\av_i})\ts_i
  = \Ad(f_i^{\av_i})\Ad(f_i^{-\av_i})\ts_i\ts_i
  = \Ad(f_i^{\av_i}f_i^{-\av_i})\ts_i^2
  = 1.
 \end{align*}
 The other fundamental relations are no more than 
 rewrites of the Verma identities (\lemmaref{lemma:Verma}).
 For example, in the case of $(a_{ij},a_{ji})=(-1,-1)$,
 the relation $\bs_i\bs_j\bs_i=\bs_j\bs_i\bs_j$
 is proved as below. Using the formula 
 $\tw(f_k^\beta)=f_k^{w(\beta)}$ ($w\in W$, $k\in I$, $\beta\in\Qv$),
 we obtain
 \begin{align*}
  \bs_i\bs_j\bs_i
  = \Ad(f_i^{\av_i} f_j^{\av_i+\av_j} f_i^{\av_j})\ts_i\ts_j\ts_i,
\quad
  \bs_j\bs_i\bs_j
  = \Ad(f_j^{\av_j} f_i^{\av_i+\av_j} f_j^{\av_i})\ts_j\ts_i\ts_j.
 \end{align*}
 Therefore the relation $\bs_i\bs_j\bs_i=\bs_j\bs_i\bs_j$
 follows from the Verma identity for $(a_{ij},a_{ji})=(-1,-1)$.
 The other relations are proved by the same argument.
 \qed
\end{proof}

This lemma immediately leads to the following theorem.

\begin{theorem}
\label{theorem:QWGA}
 The mapping $s_i\mapsto\bs_i$ ($i\in I$) induces
 the Weyl group actions on $D(\A)$ 
 (and on $D(\A_{(P)})$ when $A$ is an Ore domain).
 \qed
\end{theorem}

This theorem can be regarded as 
a both $q$-difference and canonically quantized version of 
Theorem 1.1 and Theorem 1.2 of \cite{NY0012028}.

\begin{definition}
\label{definition:QBWGA}
 We call the Weyl group actions obtained in \theoremref{theorem:QWGA}
 {\em the quantum birational Weyl group actions}
 and denote by $w(x)$ the quantum birational action of $w\in W$ on $x$.
 \qed
\end{definition}

\begin{remark}
 In \cite{Hasegawa2007}, 
 using the quantum dilogarithm, 
 Hasegawa quantizes certain birational Weyl group actions
 of $q$-difference type proposed by Kajiwara, Noumi, and Yamada \cite{KNY}.
 Although the quantum birational Weyl group actions of Hasegawa are different from ours,
 they can be also reconstructed by the method of fractional powers 
 (Section 5 of \cite{Kuroki2008}). 
 The quantum $\tau$-functions for the Hasegawa actions 
 are not discovered at the present time.
 \qed
\end{remark}


\subsection{Explicit formulas for the action}
\label{sec:explicit}

The following formulas immediately 
follow from the definition of $\bs_i$
and the quantum birational Weyl group action 
(\definitionref{definition:QBWGA}):
\begin{align}
 &
 s_i(\beta) = \beta - \bra\beta,\alpha_i\ket\av_i
 \quad (\beta\in\Qv)
 \quad \text{in the Kac-Moody case},
 \label{eq:s_i(beta)}
 \\ &
 s_i(q^\beta) = q^{s_i(\beta)} = q^{\beta - \bra\beta,\alpha_i\ket\av_i}
 \quad (\beta\in\Qv)
 \quad \text{in the $q$-difference case},
 \label{eq:s_i(q^beta)}
 \\ &
 s_i(\tau^\mu)
 = f_i^{\bra\av_i,\mu\ket}\tau^{s_i(\mu)}
 = f_i^{\bra\av_i,\mu\ket}\tau^{\mu-\bra\av_i,\mu\ket\alpha_i}
 \quad (\mu\in P),
 \label{eq:s_i(tau^mu)}
 \\ &
 s_i(f_i) = f_i.
 \label{eq:s_i(f_i)}
\end{align}
The second last formula follows from 
$\tau^\mu f_i^{\av_i}=f_i^{\av_i+\bra\av_i,\mu\ket}\tau^\mu$ and
$\bra\av_i,s_i(\mu)\ket=-\bra\av_i,\mu\ket$:
\begin{align*}
 s_i(\tau^\mu) 
 &
 = f_i^{\av_i} \tau^{s_i(\mu)} f_i^{-\av_i}
 = f_i^{\av_i}  f_i^{-\av_i+\bra\av_i,\mu\ket} \tau^{s_i(\mu)}
 = f_i^{\bra\av_i,\mu\ket} \tau^{s_i(\mu)}.
\end{align*}
In particular, we have \(
 s_i(\tau_i) = f_i\tau^{\Lambda_i-\alpha_i}
\) and \(
 s_i(\tau_j) = \tau_j
\) ($i\ne j$)..

\begin{remark}
\label{remark:s_i(f_j)}
Since $f_j=s_j(\tau_j)\tau^{-\Lambda_j+\alpha_j}$, we have
\begin{equation*}
 s_i(f_j) = s_is_j(\tau_j)s_i(\tau^{-\Lambda_j+\alpha_j})
 \quad (i,j\in I).
\end{equation*}
This means that the quantum birational Weyl group action on 
the dependent variables $f_i$ ($i\in i$) is described 
by the action on the quantum Laurent $\tau$-monomials $\tau^\mu$ ($\mu\in P$).
This observation is one of the motivation of introducing the quantum $\tau$-functions.
For the definition of them, see \secref{sec:def-tau}.
\qed
\end{remark}

We shall write down the explicit formulas of $s_i(f_j)$ for $i\ne j$ as follows.

In the Kac-Moody case, 
we define the commutator $[A,B]$ by $[A,B]=AB-BA$
and $\ad f:U(\g)\to U(\g)$ for $f\in\g$ by
\begin{equation*}
 (\ad f)(a) = fa - af \quad (a\in U(\g)).
\end{equation*}
Then the Serre relations are rewritten 
in the form $(\ad f_i)^{1-a_{ij}}(f_j) = 0$ ($i\ne j$).
More generally, we have
\begin{equation*}
 (\ad f)^k(a) = \sum_{s=0}^k (-1)^s \binom{k}{s} f^{k-s} a f^s
 \quad (k\in\Z_{\geqq0}).
\end{equation*}
By induction on $|n|$, we can obtain
\begin{equation*}
 f_i^nf_jf_i^{-n} 
 = \sum_{k=0}^{-a_{ij}}\binom{n}{k}(\ad f_i)^k(f_j)f_i^{-k}
 \quad (i\ne j, n\in\Z).
\end{equation*}
It immediately follows that
\begin{equation*}
 f_i^{\beta} f_j f_i^{-\beta}
 = \sum_{k=0}^{-a_{ij}}\binom{\beta}{k}(\ad f_i)^k(f_j)f_i^{-k}
 \in \tA^\pa
 \quad (i\ne j, \beta\in\Qv).
\end{equation*}
In particular, we have
\begin{equation}
 s_i(f_j)
 = f_i^{\av_i} f_j f_i^{-\av_i}
 = \sum_{k=0}^{-a_{ij}}\binom{\av_i}{k}(\ad f_i)^k(f_j)f_i^{-k}
 \in \tA^\pa
 \quad (i\ne j).
 \label{eq:s_i(f_j)-KM}
\end{equation}
This result is a canonically quantized version of 
the formula (1.9) of \cite{NY0012028} specialized by $\psi=\varphi_j$.

In the $q$-difference case, 
we define the $q$-commutator $[A,B]_q$ by $[A,B]_q=AB-qBA$
and $\ad f_i:U_q(\g)\to U_q(\g)$ by
\begin{equation*}
 (\ad f_i)(a) = f_i a - q_i^{-h_i} a q_i^{h_i} f_i
 \quad (a\in U_q(\g)).
\end{equation*}
Then the $q$-Serre relations are also rewritten 
in the form $(\ad f_i)^{1-a_{ij}}(f_j) = 0$ ($i\ne j$).
More generally, we have
\begin{align*}
 (\ad f_i)^k(f_j)
 &=
 [f_i,
   [
   \cdots,
     [f_i,
       [f_i,f_j]_{q_i^{a_{ij}}}
     ]_{q_i^{a_{ij}+2}}
   \cdots
   ]_{q_i^{a_{ij}+2(k-2)}}
 ]_{q_i^{a_{ij}+2(k-1)}}
 \\ 
 &= \sum_{s=0}^k 
 (-1)^s q_i^{s(k+a_{ij}-1)} 
 \qbinom{k}{s}_{q_i} f_i^{k-s} f_j f_i^s
 \quad (k\in\Z_{\geqq0}).
\end{align*}
By induction on $|n|$, we can obtain
\begin{equation*}
 f_i^nf_jf_i^{-n}
 = \sum_{k=0}^{-a_{ij}}
   q_i^{(k+a_{ij})(n-k)}\qbinom{n}{k}_{q_i}
   (\ad f_i)^k(f_j)f_i^{-k}
 \quad (i\ne j, n\in \Z).
\end{equation*}
For the derivation of this formula, see also Chapter 7 of \cite{Lusztig}.
It immediately follows that
\begin{equation*}
 s_i(f_j)
 = f_i^{\beta} f_j f_i^{-\beta}
 = \sum_{k=0}^{-a_{ij}}
   q_i^{(k+a_{ij})(\beta-k)}\qbinom{\beta}{k}_{q_i}
   (\ad f_i)^k(f_j)f_i^{-k}
 \in \tA^\pa
 \quad (i\ne j, \beta\in\Qv).
\end{equation*}
In particular, we have
\begin{equation}
 s_i(f_j)
 = f_i^{\av_i} f_j f_i^{-\av_i}
 = \sum_{k=0}^{-a_{ij}}
   q_i^{(k+a_{ij})(\av_i-k)}\qbinom{\av_i}{k}_{q_i}
   (\ad f_i)^k(f_j)f_i^{-k}
 \in \tA^\pa
 \quad (i\ne j).
 \label{eq:s_i(f_j)-q}
\end{equation}
This result is a both $q$-difference and canonically quantized version of 
the formula (1.9) of \cite{NY0012028} specialized by $\psi=\varphi_j$.

Thus we obtain the following lemma.

\begin{lemma}
\label{lemma:in-tApa}
 For any $\beta\in\Qv$ and $i,j\in I$, 
 we have $f_i^\beta f_j f_i^{-\beta}\in\tA^\pa$.
 More precisely, $f_i^{\beta}f_jf_i^{-\beta}$ belongs to
 the subalgebra of $\tA^\pa$ generated by $\{f_i^{\pm1}, f_j, \beta\}$ 
 in the Kac-Moody case and by $\{f_i^{\pm1},f_j,q_i^{\pm\beta}\}$ 
 in the $q$-difference case.
 In particular, we have $s_i(f_j)\in\tA^\pa$.
 \qed
\end{lemma}

\begin{example}
\label{example:s_i(f_j)}
Suppose that  $a_{ij}=-1$ and $d_i=1$. 
In the Kac-Moody case, we have
\begin{equation*}
 s_i(f_j) 
 = f_j + \av_i [f_i,f_j] f_i^{-1}
 = (1-\av_i) f_j + \av_i f_if_jf_i^{-1}.
\end{equation*}
In the $q$-difference case, we have
\begin{equation*}
 s_i(f_j) 
 = q^{-\av_i} f_j + [\av_i]_q[f_i,f_j]_{q^{-1}}f_i^{-1}
 = [1-\av_i]_q f_j + [\av_i]_q f_if_jf_i^{-1}.
\end{equation*}
This formula shall be used in \exampleref{example:tau(s_is_j(Lambda_j))}.
\qed
\end{example}

\begin{remark}
 Suppose that $A$ is an Ore domain.
 It follows from \lemmaref{lemma:in-tApa} that 
 the quantum birational Weyl group
 action preserves $D(A^\pa_{(P)})$ and $A^\pa_{(P)}$.
 Therefore $W$ acts on them.
 This result is an extended version of Theorem 4.3 of \cite{Kuroki2008}.
 Since the algebra $D(A^\pa_{(P)})$ does not contain
 the fractional powers $f_i^\beta$, 
 the quantum birational Weyl group action on $D(A^\pa_{(P)})$
 is characterized by the explicit formulas written down 
 in this subsection.
 \qed
\end{remark}

\begin{example}
 Under the setting of \exampleref{example:x-d-Verma}, we have
 \begin{align*}
  &
  s_1(x)  = x, \quad
  s_1(\d) = x^{\av_1}\d x^{-\av_1} = \d - \frac{\av_1}{x},
  \\ &
  s_2(x) = \d^{\av_2}x\d^{-\av_2} = x + \av_2\d^{-1}, \quad
  s_2(\d) = \d,
  \\ &
  s_i(\av_i) = -\av_i, \quad s_1(\av_2) = s_2(\av_1) = \av_1+\av_2.
 \end{align*}
 These formulas define the action of $W=S_3=\bra s_1,s_2\ket$ 
 on $A^\pa_{(P)}=\C[x,\d,\av_1,\av_2]_{(P)}$.
 A non-commutative rational function $a\in\C(x,\d,\av_1,\av_2)$ 
 is an element of $\C[x,\d,\av_1,\av_2]_{(P)}$
 if and only if, for any $\lambda\in P$,
 there exist some $b,s\in\C[x,\d,\av_1,\av_2]$ 
 such that $a=bs^{-1}$ and $s$ does not vanish 
 even if $\bra\av_i,\lambda\ket$'s are substituted into $\av_i$'s.
 \qed 
\end{example}


\section{Quantum $\tau$-functions}
\label{sec:QTF}


\subsection{Definition of quantum $\tau$-functions}
\label{sec:def-tau}

The Tits cone is defined to 
be $WP_+=\{\,w(\mu)\mid w\in W, \mu\in P_+\,\}\subset P$.
For each $\nu\in WP_+$, 
we define the quantum $\tau$-function $\tau_{(\nu)}$ by
\begin{equation*}
 \tau_{(\nu)} = w(\tau^\mu), \quad
 \nu = w(\mu), \quad w\in W, \quad \mu\in P_+.
\end{equation*}
Note that $w(\tau^\mu)$ depends only on $\nu=w(\mu)$ 
owing to the property $s_i(\tau_j)=\tau_j$ ($i\ne j$)
of the quantum birational Weyl group action.

\begin{example}
\label{example:tau(s_is_j(Lambda_j))}
In the $q$-difference case, if $i\ne j$, then we have
\begin{align*}
 &
 \tau_{(\Lambda_j)} = \tau_j, \quad
 \tau_{(s_j(\Lambda_j))} = s_j(\tau_j) = f_j \,\tau^{s_j(\Lambda_j)},
 \\ &
 \tau_{(s_is_j(\Lambda_j))} = s_is_j(\tau_j)
 = f_i^{\av_i} f_j \tau^{s_is_j(\Lambda_j)} f_i^{-\av_i}
 = f_i^{\av_i} f_j f_i^{-\av_i-a_{ij}}
   \tau^{s_is_j(\Lambda_i)} 
 \\ & 
 \hphantom{\tau_{(s_is_j(\Lambda_j))}}
 = \left(
    \sum_{k=0}^{-a_{ij}}
    q_i^{(k+a_{ij})(\av_i-k)}\qbinom{\av_i}{k}_{q_i}
    (\ad f_i)^k(f_j) f_i^{-a_{ij}-k}
  \right)
  \tau^{s_is_j(\Lambda_i)}.
\end{align*}
Note that the quantum $\tau$-function $\tau_{(s_is_j(\Lambda_j))}$ is 
a polynomial in $f_i,f_j$ and 
a Laurent polynomial in $q_i^{\av_i}$.
In particular, when $a_{ij}=-1$ and $d_i=1$, we have
\begin{align*}
 \tau_{(s_is_j(\Lambda_j))} = s_is_j(\tau_j)
 &= 
 \left(q^{-\av_i}f_jf_i + [\av_j]_q[f_i,f_j]_{q^{-1}}\right)
 \tau^{s_js_i(\Lambda_i)}
 \\
 &= 
 \left([1-\av_i]_q f_jf_i + [\av_i]_q f_if_j\right)
 \tau^{s_js_i(\Lambda_i)}.
\end{align*}
The last expression shall be used for the proof of 
the quantum $q$-Hirota-Miwa equation
in \secref{sec:QHME}.
\qed
\end{example}

\begin{example}
Assume that the GCM is of type $A_3$:
$I=\{1,2,3\}$, $a_{ii}=2$, $a_{i,i+1}=a_{i+1,i}=-1$ ($i=1,2$), 
$a_{ij}=0$ ($|i-j|\geqq2$), $d_i=1$.
Let $A$ be the associative algebra over $\C$ generated by $f_1,f_2,f_3$
with fundamental relations 
$[f_1,f_2]=[f_2,f_3]=1$, $[f_1,f_3]=0$.

We set $(i_1,i_2,\ldots,i_6)=(1,2,3,1,2,1)$, 
$w_k=s_{i_k}\cdots s_{i_2}s_{i_1}\in W$, and \(
 \beta_k := w_{k-1}^{-1}(\av_{i_k})
 = s_{i_1}\cdots s_{i_{k-1}}(\av_{i_k})
\) ($k=1,2,\ldots,6$).
We have 
$\beta_1=\av_1$, 
$\beta_2=\av_1+\av_2$,
$\beta_3=\av_1+\av_2+\av_3$,
$\beta_4=\av_2$,
$\beta_5=\av_2+\av_3$, and
$\beta_6=\av_3$.

Then the quantum $\tau$-functions 
$\tau_{(w_k(\Lambda_1))}$ ($k=1,2,\ldots,6$) can be written in the form 
$\tau_{(w_k(\Lambda_1))}=\tw_k(X_k)\tau^{w_k(\Lambda_1)}$,
where $X_k$ ($k=1,2,\ldots,6$) are calculated as follows:
\begin{align*}
 X_1 &= f_1^{-\beta_1}f_1^{\beta_1+1} = f_1,
 \\ 
 X_2 &= f_2^{-\beta_2}X_1f_2^{\beta_2+1} 
 = \left(f_1+\tfrac{\beta_2}{f_2}\right)f_2 = f_1f_2+\beta_2,
 \\
 X_3 &= f_3^{-\beta_3}X_2f_3^{\beta_3+1}
 = \left(f_1\left(f_2+\tfrac{\beta_3}{f_3}\right)+\beta_2\right)f_3
 = f_1f_2f_3 + \beta_3f_1 + \beta_2f_3,
 \\ 
 X_4 &= f_1^{-\beta_4}X_3f_1^{\beta_4}
 = f_1\left(f_2-\tfrac{\beta_4}{f_1}\right)f_3+\beta_3f_1+\beta_2f_3
 = f_1f_2f_3 + \beta_3f_1 + (\beta_2-\beta_4)f_3,
 \\
 X_5 &= f_2^{-\beta_5}X_4f_2^{\beta_5} 
 = \left(f_1+\tfrac{\beta_5}{f_2}\right)f_2\left(f_3-\tfrac{\beta_5}{f_2}\right)
 + \beta_3\left(f_1+\tfrac{\beta_5}{f_2}\right)
 + (\beta_2-\beta_4)\left(f_3-\tfrac{\beta_5}{f_2}\right)
 \\
 &= f_1f_2f_3 + (\beta_3-\beta_5)f_1 + (\beta_2-\beta_4+\beta_5)f_3
 + (\underbrace{-\beta_2+\beta_4-\beta_5+\beta_3}_{\text{cancels out}})
   \tfrac{\beta_5}{f_2},
 \\[-\bigskipamount]
 &= f_1f_2f_3 + (\beta_3-\beta_5)f_1 + (\beta_2-\beta_4+\beta_5)f_3,
 \\
 X_6 &= f_1f_2f_3 + \beta_6 f_1 + (\beta_3-\beta_6)f_3.
\end{align*}
Thus the all quantum $\tau$-functions 
$\tau_{(w_k(\Lambda_1))}$ ($k=1,2,\ldots,6$) are polynomials in 
$f_i$ and $\av_i$ ($i\in I$).
\qed
\end{example}


\subsection{The quantum $q$-Hirota-Miwa equation}
\label{sec:QHME}

In this subsection, as a supporting evidence for the correctness of
the definition of the quantum $\tau$-functions in the previous subsection, 
we shall show that the quantum $\tau$-functions of type $A^{(1)}_{n-1}$ satisfy
the quantum $q$-Hirota-Miwa equation \eqref{eq:QHME} in the $q$-difference case.

Assume that $n\geqq 3$ and the GCM $[a_{ij}]_{i,j\in I}$ is of type $A^{(1)}_{n-1}$:
$I=\Z/n\Z$, $a_{ii}=2$, $a_{i,i\pm1}=-1$, $a_{ij}=0$ ($j\ne i,i\pm1$),
and $d_i=1$.
Denote the image of $k\in\Z$ in $I=\Z/n\Z$ by $\overline{k}$.
Assume that the algebra automorphism of $A$ can be defined by 
$f_i\mapsto f_{i+1}$ for $i\in I$.

We define the coroot lattice $\Qv$ to be the free $\Z$-module generated 
by $\deltav$ and $\epsv_k$ ($k=1,2,\ldots,n$).
We set $\epsv_k$ ($k\in\Z$) by the quasi-periodicity $\epsv_{k+n}=\epsv_k-\deltav$.
Define the simple coroots by $\av_k=\epsv_k-\epsv_{k+1}$ ($k\in\Z$). 
Then we have $\av_{k+n}=\av_k$ and put $\av_{\overline{k}}=\av_k$. 
Since $\av_0=\deltav+\epsv_n-\epsv_1$, 
the set $\{\av_k\}_{k=0}^{n-1}$ is linearly independent over $\Z$
and $\sum_{k=0}^{n-1}\av_k=\deltav$.

The weight lattice $P$ is given by $P=\Hom(\Qv,\Z)$.
Denote by $\Lambda_0$, $\eps_k$ ($k=1,2,\ldots,n$)
the dual basis of $\deltav$, $\epsv_k$ ($k=1,2,\dots,n$).
We set $\eps_k$ ($k\in\Z$) by the periodicity $\eps_{k+n}=\eps_k$.
We define $\varpi_k$ ($k\in\Z$) by
$\varpi_k=\eps_1+\eps_2+\cdots+\eps_k$ ($k\in\Z_{\geqq0}$)
and the quasi-periodicity $\varpi_{k+n}=\varpi_k+\varpi_n$.
Then $\Lambda_0$ and $\varpi_k$ ($k=1,2,\ldots,n$) span the weight lattice $P$.
We define $\Lambda_k$ ($k\in\Z$) by
$\Lambda_k=\Lambda_0+\varpi_k$.
Then $\Lambda_{k+n}=\Lambda_k+\varpi_n$.
Since $\bra\av_k,\varpi_n\ket=0$ ($k\in\Z$), 
we have $\bra\av_k,\Lambda_l\ket=\delta_{\overline{k},\overline{l}}$.

We define the simple roots $\alpha_k$ ($k\in\Z$) by \(
 \alpha_k
 =-\Lambda_{k-1}+2\Lambda_k-\Lambda_{k+1}
 = \eps_k-\eps_{k+1}
\). Then we have $\alpha_{k+n}=\alpha_k$
and $\bra\av_k,\alpha_l\ket=a_{\overline{k},\overline{l}}$. 
Put $\alpha_{\overline{k}}=\alpha_k$.
In this setting, we have $\sum_{k=0}^{n-1}\alpha_k=0$.

Set $s_k=s_{\overline{k}}$ and $f_k=f_{\overline{k}}$ for $k\in\Z$.
The Weyl group actions on $\Qv$ and $P$ are characterized by
the following formulas:
\begin{align*}
 &
 s_k(\deltav)=\deltav, \quad
 s_k(\epsv_k) = \epsv_{k+1}, \quad
 s_k(\epsv_{k+1}) = \epsv_k, \quad
 s_k(\epsv_l)=\epsv_l \quad (\overline{l}\ne\overline{k},\overline{k+1}), 
 \\ &
 s_0(\Lambda_0) 
 = \Lambda_{-1}-\Lambda_0+\Lambda_1
 = \Lambda_0 - \eps_n + \eps_1,
 \quad
 s_k(\Lambda_0) = \Lambda_0 \quad (\overline{k}\ne\overline{0}), 
 \\ &
 s_k(\eps_k) = \eps_{k+1}, \quad
 s_k(\eps_{k+1}) = \eps_k, \quad
 s_k(\eps_l)=\eps_l \quad (\overline{l}\ne\overline{k},\overline{k+1}).
\end{align*}

We define the $\tau$-variables $\tau_k$ ($k\in\Z$) by $\tau_k=\tau^{\Lambda_k}$.
Then $\tau_{k+n}=\tau_k\tau^{\varpi_n}$. 
Note that $\tau^{\varpi_n}$ commutes with all $q^{\av_k}$.

\begin{lemma}
\label{lemma:QHME}
  For any $k\in\Z$, we have
 \[
   [\av_{k+1}]_q       \tau_k\,s_k s_{k+1}(\tau_{k+1})
  +[\av_k]_q           s_{k+1}s_k(\tau_k)\,\tau_{k+1}
  =[\av_k+\av_{k+1}]_q s_k(\tau_k)\,s_{k+1}(\tau_{k+1}).
 \]
\end{lemma}

\begin{proof}
 The definition of the quantum birational Weyl group action
 immediately leads to the following formulas:
 \begin{align*}
  &
  s_k(\tau_k) = f_k \frac{\tau_{k-1}\tau_{k+1}}{\tau_k}, 
  \quad
  s_{k+1}(\tau_{k+1}) = f_{k+1} \frac{\tau_k\tau_{k+2}}{\tau_{k+1}},
  \\ &
  s_k s_{k+1}(\tau_{k+1}) =
  \left([1-\av_k]_qf_{k+1}f_k + [\av_k]_q f_k f_{k+1}\right)
  \frac{\tau_{k-1}\tau_{k+2}}{\tau_k},
  \\ &
  s_{k+1}s_k(\tau_k) =
  \left([1-\av_{k+1}]_q f_k f_{k+1} + [\av_{k+1}]_q f_{k+1}f_k\right)
  \frac{\tau_{k-1}\tau_{k+2}}{\tau_{k+1}}.
 \end{align*}
 See \exampleref{example:tau(s_is_j(Lambda_j))}.
 Using $\tau_k q^{\av_k} = q^{\av_k+1} \tau_k$, we obtain
 \begin{align*}
  &
  [\av_{k+1}]_q \tau_k s_k s_{k+1}(\tau_{k+1}) =
  [\av_{k+1}]_q \left([-\av_k]_q f_{k+1}f_k + [\av_k+1]_q f_k f_{k+1}\right)
  \tau_{k-1}\tau_{k+2},
  \\ & 
  [\av_k]_q s_{k+1}s_k(\tau_k)\tau_{k+1} =
  [\av_k]_q \left([1-\av_{k+1}]_q f_k f_{k+1} + [\av_{k+1}]_q f_{k+1}f_k\right)
  \tau_{k-1}\tau_{k+2}.
 \end{align*}
 Add the right-hand sides of the two formulas. 
 Then the $f_{k+1}f_k$-terms cancel out and we get
 \begin{equation*}
  [\av_k+\av_{k+1}]_q f_k f_{k+1}\tau_{k-1}\tau_{k+2}
  = [\av_k+\av_{k+1}]_q s_k(\tau_k) s_{k+1}(\tau_{k+1}).
 \end{equation*}
 The lemma has been proved. 
 \qed
\end{proof}

\begin{remark}
\label{remark:QHME}
In the quantum case, we must be careful to non-commutativity.
In the above lemma, 
we have the following commutativity and non-commutativity:
\begin{enumerate}
\item Although $\tau_k$ and $s_k s_{k+1}(\tau_{k+1})$ does not commute,
each of them commutes with $[\av_{k+1}]_q$.

\item Although $s_{k+1}s_k(\tau_k)$ and $\tau_{k+1}$ does not commute,
each of them commutes with $[\av_k]_q$.

\item $s_k(\tau_k)$ and $s_{k+1}(\tau_{k+1})$ commutes.
Although each of them does not commutes with $[\av_k+\av_{k+1}]_q$,
their product $s_k(\tau_k)s_{k+1}(\tau_{k+1})$ commutes with $[\av_k+\av_{k+1}]_q$.
\qed
\end{enumerate}
\end{remark}

The extended Weyl group $\tW$ is defined by the semi-direct product
$\tW = W\rtimes\bra\pi\ket$ with defining relations 
$\pi s_k = s_{k+1} \pi$ ($k\in\Z$).
The actions of $\pi$ on $\Qv$ and $P$ are given by
$\pi(\deltav)=\deltav$,
$\pi(\epsv_k)=\epsv_{k+1}$,
$\pi(\Lambda_0)=\Lambda_1=\Lambda_0+\eps_1$, and 
$\pi(\eps_k)=\eps_{k+1}$.
These formulas define the extended Weyl group action on $\Qv$ and $P$
which preserves the canonical pairing between them.
Then we have $\pi(\av_k)=\av_{k+1}$, 
$\pi(\varpi_k)=\varpi_{k+1}-\eps_1$, and
$\pi(\Lambda_k)=\Lambda_{k+1}$.

We define the actions of $\pi$ on $f_k$ and $\tau_k$ by
$\pi(f_k)=f_{k+1}$ and $\pi(\tau_k)=\tau_{k+1}$.
Then we have $\pi(\tau^\mu)=\tau^{\pi(\mu)}$ ($\mu\in P$).

We define $T_k\in\tW$ ($i\in \Z$) by 
$T_k=s_{k-1}\cdots s_2s_1\pi s_{n-1}s_{n-2}\cdots s_k$
($k=1,2,\ldots,n$) and the periodicity $T_{k+n}=T_k$.
Then $T_k$ ($k\in\Z$) mutually commute and we have
\begin{align*}
 &
 s_k T_k s_k^{-1} = T_{k+1}, \quad
 s_k T_{k+1} s_k^{-1} = T_k, \quad
 s_k T_l s_k^{-1} = T_l \quad (\overline{l}\ne\overline{k},\overline{k+1}), \quad
 \pi T_k \pi^{-1}= T_{k+1},
 \\ & 
 T_k(\deltav)=\deltav, \quad
 T_k(\epsv_l)=\epsv_l-\delta_{\overline{k},\overline{l}}\deltav, \quad
 T_k(\eps_l)=\eps_l, \quad
 T_k(\Lambda_0)=\Lambda_0+\eps_k.
\end{align*} 
For $m=\sum_{k=1}^n m_k\eps_k\in L=\bigoplus_{k=1}^n\Z\eps_k$,
we put $T^m = \prod_{k=1}^n T_k^{m_k}$.
Then we have $T^m(\Lambda_k)=\Lambda_k+m$.

We define $\epsv_k(m)$, $\av_k(m)$, and $\tau_k(m)$ for $m\in L$ by
\begin{align*}
 &
 \epsv_k(m) = T^m(\epsv_k) = \epsv_k - m_k\deltav, 
 \quad
 \av_k(m)=T^m(\av_k)=\av_k+(m_{k+1}-m_k)\deltav,
 \\ &
 \tau_k(m) = T^m(\tau_k) = \tau_{(\Lambda_k+m)}.
\end{align*}
Here we assume that $m_{k+n}=m_k$ for $k\in\Z$.
Then we have
\begin{alignat*}{2}
&
\Lambda_k = \Lambda_{k-1}+\eps_k, \quad
& &
\Lambda_{k+1} = \Lambda_{k-1}+\eps_k+\eps_{k+1}, 
\\ &
s_k(\Lambda_k) = \Lambda_{k-1}+\eps_{k+1}, \quad
& &
s_{k+1}s_k(\Lambda_{k+1}) = \Lambda_{k-1}+\eps_{k+2},
\\ &
s_{k+1}(\Lambda_{k+1}) = \Lambda_{k-1}+\eps_k+\eps_{k+2}, \quad
& &
s_k s_{k+1}(\Lambda_{k+1}) = \Lambda_{k-1}+\eps_{k+1}+\eps_{k+2}.
\end{alignat*}
Therefore, applying $T^m$ to the both-hand sides of the formula in \lemmaref{lemma:QHME},
we obtain
\begin{align*}
 &
   [\av_{k+1}(m)]_q \tau_{k-1}(m+\eps_k)     \tau_{k-1}(m+\eps_{k+1}+\eps_{k+2})
 \\ & \qquad
 + [\av_k(m)]_q     \tau_{k-1}(m+\eps_{k+2}) \tau_{k-1}(m+\eps_k+\eps_{k+1})
 \\ & \qquad\qquad
 = [\av_k(m)+\av_{k+1}(m)]_q \tau_{k-1}(m+\eps_{k+1}) \tau_{k-1}(m+\eps_k+\eps_{k+2}).
\end{align*}
This equation is rewritten in the following cyclically symmetric form:
\begin{align}
 &
   [\epsv_{k+1}(m)-\epsv_{k+2}(m)]_q \tau_{k-1}(m+\eps_k)     \tau_{k-1}(m+\eps_{k+1}+\eps_{k+2})
 \notag
 \\ & \qquad
 + [\epsv_k(m)-\epsv_k(m)]_q          \tau_{k-1}(m+\eps_{k+2}) \tau_{k-1}(m+\eps_k+\eps_{k+1})
 \label{eq:QHME}
 \\ & \qquad\qquad
 + [\epsv_{k+2}(m)-\epsv_k(m)]_q     \tau_{k-1}(m+\eps_{k+1}) \tau_{k-1}(m+\eps_k+\eps_{k+2})
 =0.
 \notag
\end{align}
We call this equation {\em the quantum $q$-Hirota-Miwa equation}.
The original (non-quantum) Hirota-Miwa equation is found in  
Equation (2.1) of \cite{Hirota} and in Equation (2.6) of \cite{Miwa}.
Although the method for the above derivation is same as 
the one in Section 4.5 of \cite{Noumi},
we must be careful to the non-commutativity mentioned in \remarkref{remark:QHME}.


\subsection{Definition of regularity of the quantum $\tau$-functions}
\label{sec:def-reg}

\begin{definition}
 For each $\nu\in WP_+$, the quantum $\tau$-function $\tau_{(\nu)}$ 
 is said to be {\em regular} if $\tau_{(\nu)}\in A^\pa$, namely, 
 if $\tau_{(\nu)}$ is a polynomial in $\{f_i,\av_i\}_{i\in I}$
 for the Kac-Moody case, and 
 if $\tau_{(\nu)}$ is a polynomial in $\{f_i,q^{\pm\av_i}\}_{i\in I}$
 for the $q$-difference case.
 \qed
\end{definition}

As mentioned in the introduction, the regularity of the classical $\tau$-functions 
is shown by Noumi and Yamada in \cite{NY0012028}.
The rest of this paper is devoted to 
the proof of the regularity of the quantum $\tau$-functions.

It is sufficient for the proof of the regularity of the quantum $\tau$-functions for any $A$
to obtain the regularity for $A=U(\n_-)$ in the Kac-Moody case
and for $A=U_q(\n_-)$ in the $q$-difference case.
Therefore, in the following subsections, we set
$A=U_-=U(\n_-)$, 
$A^\pa=U_-^\pa=U_-[\beta|\beta\in\Qv]$, 
$\tA^\pa=\tU_-^\pa=U_-[f_i^{-1},\beta|i\in I,\beta\in\Qv]$, 
$U=U(\g)$
in the Kac-Moody case, 
and
$A=U_-=U_q(\n_-)$, 
$A^\pa=U_-^\pa=U_-[q^\beta|\beta\in\Qv]$, 
$\tA^\pa=\tU_-^\pa=U_-[f_i^{-1},q^\beta|i\in I,\beta\in\Qv]$, 
$U=U_q(\g)$
in the $q$-difference case.
Moreover we assume that 
$\{\alpha_i\}_{i\in I}$ is also linearly independent over $\Z$.


\subsection{Relation to singular vectors in Verma modules}
\label{sec:sing}

In the following we fix $\mu\in P_+$ and assume that 
$w_n = s_{i_n}\cdots s_{i_2}s_{i_1}$ is a reduced expressions of $w_n\in W$
for each $n=0,1,\ldots,N$.
We set 
$\beta_n = w_{n-1}^{-1}(\av_{i_n})=s_{i_1}s_{i_2}\cdots s_{i_{n-1}}(\av_{i_n})$
for $n=1,2,\ldots,N$.

By the definition of the quantum birational Weyl group action 
(\definitionref{definition:QBWGA}), 
for $n=0,1,\ldots,N$, we have
\begin{align*}
  \tw_n^{-1}\tau_{(w_n(\mu))} 
  &
  =\tw_n^{-1}
    \Ad(f_{i_n}^{\av_{i_n}})\ts_{i_n}\cdots
    \Ad(f_{i_2}^{\av_{i_2}})\ts_{i_2}
    \Ad(f_{i_1}^{\av_{i_1}})\ts_{i_i}(\tau^\mu)
  \\ &
  = 
    \Ad(f_{i_n}^{-\beta_n})\cdots
    \Ad(f_{i_2}^{-\beta_2})
    \Ad(f_{i_1}^{-\beta_1})(\tau^\mu)
  \\ &
  = 
    f_{i_n}^{-\beta_n}\cdots f_{i_2}^{-\beta_2}f_{i_1}^{-\beta_1}
    \tau^\mu
    f_{i_1}^{\beta_1} f_{i_2}^{\beta_2} \cdots  f_{i_n}^{\beta_n}
  \\ &
  = 
    f_{i_n}^{-\beta_n}\cdots 
    f_{i_2}^{-\beta_2}
    f_{i_1}^{-\beta_1}
    f_{i_1}^{\beta_1+\bra\beta_1,\mu\ket}
    f_{i_2}^{\beta_2+\bra\beta_2,\mu\ket}\cdots 
    f_{i_n}^{\beta_n+\bra\beta_n,\mu\ket}
    \tau^\mu.
\end{align*}
Thus we obtain
\begin{equation*}
 \tw_n^{-1}\tau_{(w_n(\mu))} = \Phi_n^{-1}\Psi_n\tau^\mu, 
\end{equation*}
where $\Phi_n$ and $\Psi_n$ are given by
\begin{equation*}
 \Phi_n = 
 f_{i_1}^{\beta_1}
 f_{i_2}^{\beta_2}\cdots
 f_{i_n}^{\beta_n},
 \quad
 \Psi_n = 
 f_{i_1}^{\beta_1+\bra\beta_1,\mu\ket}
 f_{i_2}^{\beta_2+\bra\beta_2,\mu\ket}\cdots
 f_{i_n}^{\beta_n+\bra\beta_n,\mu\ket}.
\end{equation*}
This is equivalent to 
$\tau_{(w_n(\mu))} = \tw_n(\Phi_n^{-1}\Psi_n)\tau^{w_n(\mu)}$.

\begin{remark}
 In general, for $w\in W$ and $\mu\in P_+$, 
 there exists a unique $\phi_w(\mu)\in\A$ with
 $\tau_{(w(\mu))}=\phi_w(\mu)\tau^{w(\mu)}$.
 In \cite{NY9708018}, the classical version of $\phi_w(\mu)$ is called the $\tau$-cocycle.
 In the quantum case, $\phi_w(\mu)$ does not commute with $\tau^\mu$ in general.
 This is the reason why we does not deal with $\phi_w(\mu)$ but $\tau_{(w(\mu))}$. 
 \qed
\end{remark}

Let $\sigma:U^-\to U^-$ be the anti-algebra involution of $U_-$ 
which sends $f_i$ to $f_i$ ($i\in I$).
That is, the linear transformation $\sigma$ reverses 
the order of products of $f_i$'s in $U_-$.
Denote the unique extension of $\sigma$ to the anti-algebra involution 
of $\tU_-=U_-[f_i^{-1}|i\in I]$
by the same symbol.

Assume that $\lambda\in P$.
The Verma module $M(\lambda)$ is defined to be the left $U$-module 
generated by $v_\lambda$ with fundamental relations
$e_iv_\lambda=0$ ($i\in I$) , 
$h_iv_\lambda=\bra\av_i,\lambda\ket v_\lambda$ ($i\in I$)
in the Kac-Moody case and
$e_iv_\lambda=0$ ($i\in I$) , 
$q^{h_i}v_\lambda=q^{\bra\av_i,\lambda\ket}v_\lambda$ ($i\in I$)
in the $q$-difference case.
The vector $v_\lambda$ is called the highest weight vector of $M(\lambda)$.

For $\lambda\in P$, 
the highest weight simple module $L(\lambda)$ is defined to be
the unique simple quotient of the Verma module $M(\lambda)$.
The highest weight vector $u_\lambda$ of $L(\lambda)$
is defined to be the image of $v_\lambda\in M(\lambda)$ in $L(\lambda)$.
The simple module $L(\lambda)$ is integrable if and only if $\lambda\in P_+$.

Recall that the Weyl vector $\rho\in P_+$ satisfies $\bra\av_i,\rho\ket=1$ ($i\in I$).
Define the shifted action of the Weyl group on $P$ by
$w\circ\lambda = w(\lambda+\rho)-\rho$ ($w\in W$, $\lambda\in P$).

Assume that $\lambda\in P_+$ and $w\in W$.
We define $F_{w,\lambda}\in U_-$ by
\begin{equation*}
 F_{w,\lambda}
 =
 f_{j_m}^{\bra\av_{j_m},s_{m_{l-1}}\cdots s_{j_2}s_{j_1}\circ\lambda\ket+1}
 \cdots
 f_{j_2}^{\bra\av_{j_2},s_{j_1}\circ\lambda\ket+1}
 f_{j_1}^{\bra\av_{j_1},\lambda\ket+1},
\end{equation*}
where $w=s_{j_m}\cdots s_{j_2}s_{j_1}$ is a reduced expression of $w$.
Note that $F_{w,\lambda}$ is independent on the choice of the reduced
expression of $w$ owing to the Verma identities, 
and $F_{w,\lambda}v_\lambda$ is a singular vector 
with weight $w\circ\lambda$ in the Verma module $M(\lambda)$,
which is unique up to scalar multiples.
(For the uniqueness, see 4.4.15 of \cite{Jos-1995}.)
For $\lambda\in P_+$, 
the kernel of the canonical projection from $M(\lambda)$ onto $L(\lambda)$
is generated by \(
 \{F_{s_i,\lambda}v_\lambda=f_i^{\bra\av_i,\lambda\ket+1}v_\lambda\}_{i\in I}
\).

Using 
\(
 \bra\beta_k,\lambda+\rho\ket
 = \bra\av_{i_k},s_{i_{k-1}}\cdots s_{i_2}s_{i_1}\circ(\lambda)\ket+1
\) for $k=0,1,\ldots,N$, we obtain
\begin{align*}
 \sigma(\phi_{\lambda+\rho}(\Phi_n)) = F_{w_n,\lambda},
 \quad
 \sigma(\phi_{\lambda+\rho}(\Psi_n)) = F_{w_n,\lambda+\mu}.
\end{align*}
The quantum $\tau$-function $\tau_{(w_n(\mu))}$ 
and the singular vectors 
$F_{w_n,\lambda}v_\lambda\in M(\lambda)$, 
$F_{w_n,\lambda+\mu}v_{\lambda+\mu}\in M(\lambda+\mu)$
are related in this way.

For each $i\in I$, the multiplicative subset of $U_-$ generated 
by the single $f_i$ is an Ore set in $U_-$ 
owing to the Serre and $q$-Serre relations. 
Therefore we obtain the localization $U_-[f_i^{-1}]\subset\tU_-$ 
of $U_-$ with respect to it.

For each $i\in I$, we set 
$U_-[f_i^{-1}]^\pa=U_-[f_i^{-1}][\beta|\beta\in\Qv]$ in the Kac-Moody case and
$U_-[f_i^{-1}]^\pa=U_-[f_i^{-1}][q^\beta|\beta\in\Qv]$ in the $q$-difference case.
Then $U_-[f_i^{-1}]^\pa$ is a subalgebra of $\tU_-^\pa$.

For each $\lambda\in P$, the algebra homomorphism 
$\phi_\lambda:U_-[f_i^{-1}]^\pa\to U_-[f_i^{-1}]$ is defined to 
be the restriction of $\phi_\lambda:\tU_-^\pa=\tA^\pa\to\tA=\tU_-$
on $U_-[f_i^{-1}]^\pa$.

\begin{lemma}
\label{lemma:U_-[f_i^{-1}]^pa}
 For any $a\in U_-[f_i^{-1}]^\pa$, 
 if $\phi_{\lambda+\rho}(a)\in U_-$ for all $\lambda\in P_+$,
 then $a\in U_-^\pa$.
\end{lemma}

\begin{proof}
 Let $U_-[i]$ be the subalgebra of $U_-$ generated by
 $\sigma((\ad f_i)^k(f_j))$ for $j\in I\setminus\{i\}$
 and $k=0,1,\ldots,-a_{ij}$.
 Then $U_- = \bigoplus_{k=0}^\infty U_-[i]f_i^k$
 (Section 38.1 of \cite{Lusztig}).
 Therefore we have $U_-[f_i^{-1}]=\bigoplus_{k\in\Z}U_-[i]f_i^k$.
 Put $U_-[i]^\pa=U_-[i][\av_j|j\in I]$ in the Kac-Moody case
 and $U_-[i]^\pa=U_-[i][q^\beta|\beta\in\Qv]$ in the $q$-difference case.
 Then we have $U_-[f_i^{-1}]^\pa=\bigoplus_{k\in\Z} U_-[i]^\pa f_i^k$.
 Therefore any $a\in U_-[f_i^{-1}]^\pa$ is uniquely expressed as
 $a=\sum_{k\in\Z}a_kf_i^k$ where all $a_k\in U_-[i]^\pa$ ($k\in \Z$) are
 zero except for finite number of $k$.
 Then $\phi_\lambda(a)=\sum_{k\in\Z} \phi_\lambda(a_k)f_i^k$.
 Therefore $\phi_{\lambda+\rho}(a)\in U_-$ for all $\lambda\in P_+$
 implies $a_k=0$ for all $k<0$.
 Thus we obtain $a\in U_-^\pa$.
 \qed
\end{proof}

\begin{lemma}
\label{lemma:F/F}
 If $F_{w_n,\lambda+\mu}\in U_-F_{w_n,\lambda}$ 
 for all $\lambda\in P_+$ and $n=0,1,2,\ldots,N$, 
 then the quantum $\tau$-function $\tau_{(w_n(\mu))}$ 
 is regular for each $n=0,1,\ldots,N$.
\end{lemma}

\begin{proof}
 The quantum $\tau$-function $\tau_{(w_n(\mu))}$ is regular
 if and only if $\Phi_n^{-1}\Psi_n\in U_-^\pa$.
 Inductively we assume that $1\leqq n\leqq N$ 
 and $\Phi_{n-1}^{-1}\Psi_{n-1}\in U_-^\pa$.
 \lemmaref{lemma:in-tApa} leads to
 \begin{equation*}
  \Phi_n^{-1}\Psi_n
  = f_{i_n}^{-\beta_n}\Phi_{n-1}^{-1}\Psi_{n-1}f_{i_n}^{\beta_n+\bra\beta_n,\mu\ket}
  \in U_-[f_{i_n}^{-1}]^\pa.
 \end{equation*}
 We have \(
  \sigma(\phi_{\lambda+\rho}(\Phi_n^{-1}\Psi_n))
  = F_{w_n,\lambda+\mu}F_{w_n,\lambda}^{-1}
 \).
 By \lemmaref{lemma:U_-[f_i^{-1}]^pa}, 
 if $F_{w_n,\lambda+\mu}\in U_-F_{w_n,\lambda}$ for all $\lambda\in P_+$, 
 then $\Phi_n^{-1}\Psi_n\in U_-^\pa$.
 \qed
\end{proof}

\lemmaref{lemma:F/F} immediately leads to the following proposition.

\begin{prop}
\label{prop:reduce}
 If $F_{w,\lambda+\mu}\in U_- F_{w,\lambda}$ for all $\lambda,\mu\in P_+$
 and $w\in W$, then the all quantum $\tau$-functions $\tau_{(w(\mu))}$ 
 are regular. 
 \qed
\end{prop}

In this way, we can reduce the regularity of the quantum $\tau$-functions
to the divisibility (from the right) of $F_{w,\lambda+\mu}$ 
by $F_{w,\lambda}$ for $\lambda,\mu\in P_+$ and $w\in W$.


\subsection{Proof of regularity in the Kac-Moody case}
\label{sec:reg-KM}

In this subsection, 
we assume that $A=U_-=U(\n_-)$
and shall use the integral weight part ($P$-part) of the results of \cite{DGK}
on the BGG category for the Kac-Moody algebra $\g$.

A left $\g$-module $M$ is said to be integrally $\h$-diagonalizable 
if $M=\bigoplus_{\nu\in P}M_\nu$ 
where $M_\nu=\{\,v\in M\mid h_iv=\bra\av_i,\nu\ket v\ (i\in I) \,\}$ for $\nu\in P$.
We call $M_\nu$'s the weight subspaces of $M$ 
and $\nu$ a weight of $M$ if $M_\nu\ne 0$.

Let $\Oint$ be the category of left $\g$-modules $M$ satisfying 
the following conditions:
\begin{enumerate}
 \item[(A)]
  $M$ is integrally $\h$-diagonalizable with finite-dimensional weight spaces;
 \item[(B)]
  There exists finitely many weights $\mu_1,\ldots,\mu_n\in P$
  such that any weight of $M$ belongs to $\bigcup_{k=1}^n (\mu_k-Q_+)$.
\end{enumerate}
Then we have the Verma module $M(\lambda)$ and its simple quotient
$L(\lambda)$ are objects of $\Oint$ if $\lambda\in P$.
The simple module $L(\lambda)$ is integrable if and only if $\lambda\in P_+$.

Define the subset $\Kintg$ of the weight lattice $P$ by 
\begin{equation*}
 \Kintg
 = W\circ P_+ 
 = W(P_++\rho)-\rho
 = \{\, w\circ\lambda = w(\lambda+\rho)-\rho \mid w\in W, \lambda\in P_+\, \}.
\end{equation*}
Let $\Ointg$ be the full subcategory of $\Oint$ consisting of 
$\Oint$-objects all simple subquotients of which are 
isomorphic to $L(\nu)$ for some $\nu\in\Kintg$.
Then the Verma modules $M(w\circ\lambda)$ for $w\in W$ and $\lambda\in P_+$ 
are objects of $\Ointg$. 

For $\lambda\in P_+$, 
let $\O_\lambda$ be the full subcategory of $\Ointg$ consisting of 
$\Ointg$-objects all simple subquotients of which are 
isomorphic to $L(w\circ\lambda)$ for some $w\in W$.
Then any $\Ointg$-object decomposes uniquely 
as a direct sum of $\O_\lambda$-objects for $\lambda\in P_+$
(the integral part of Theorem 5.7 of \cite{DGK}).
For any $\Ointg$-object $M$, 
denote by $\pr_\lambda(M)$ the $\O_\lambda$-component of $M$.

For $\mu,\lambda\in P_+$, the translation functor 
$T^{\lambda+\mu}_\mu:\O_\lambda\to\O_{\lambda+\mu}$ is
defined by $T^\lambda_{\lambda+\mu}(M)=\pr_{\lambda+\mu}(M\otimes L(\mu))$
for $M\in\Ob\O_\lambda$. This functor is exact.
The following lemma is the integral part of Theorem 5.13 of \cite{DGK}.

\begin{lemma}
\label{lemma:TP-KM}
 For $\lambda,\mu\in P_+$ and $w\in W$, 
 the $\g$-module $T_\lambda^{\lambda+\mu}(M(w\circ\lambda))$
 is isomorphic to $M(w\circ(\lambda+\mu))$.
 \qed
\end{lemma}

For details of the theory of translation functor for the Kac-Moody
Lie algebras, see \cite{DGK}. (See also Section 2 of \cite{KW}.)

The following theorem is a canonically quantized version 
of Theorem 1.3 of \cite{NY0012028}.

\begin{theorem}
\label{theorem:regularity-KM}
 In the Kac-Moody case, 
 the all quantum $\tau$-functions $\tau_{(\nu)}$ for $\nu\in WP_+$
 are regular. 
 More precisely they are polynomials in $\{f_i,\av_i\}_{i\in I}$.
\end{theorem}

\begin{proof}
 By \propref{prop:reduce} it is sufficient to show
 $F_{w,\lambda+\mu}\in U_- F_{w,\lambda}$
 for $\lambda,\mu\in P_+$ and $w\in W$.
 Fix $\lambda,\mu\in P_+$ and $w\in W$.
 Denote $T_\lambda^{\lambda+\mu}$ by $T$.

 Since $F_{w,\lambda}v_\lambda$ is 
 a singular vector with weight $w\circ\lambda$,
 we can identify $M(w\circ\lambda)$ 
 with $U_-F_{w,\lambda}v_\lambda\subset M(\lambda)$.
 We can also identify $M(\lambda+\mu)$ 
 with $U_-(v_\lambda\otimes u_\mu)\subset M(\lambda)\otimes L(\mu)$,
 and $M(w\circ(\lambda+\mu))$ 
 with $\subset M(\lambda+\mu)$.

 By the exactness of the translation functor, we can regard   
 $T(M(w\circ\lambda))$ as a submodule of $T(M(\lambda))$.
 By \lemmaref{lemma:TP-KM} and 
 the uniqueness (up to scalar multiples) of the non-zero homomorphism 
 from $M(w\circ(\lambda+\mu))$ to $M(\lambda+\mu)$, 
 we obtain \(
   T(M(w\circ\lambda))=M(w\circ(\lambda+\mu))
   =U_-F_{w,\lambda+\mu}(v_\lambda\otimes u_\mu)
 \).  Therefore
 \begin{equation*}
  F_{w,\lambda+\mu}(v_\lambda\otimes u_\mu)
  \in M(w\circ(\lambda+\mu)) 
  = T(M(w\circ\lambda)) 
  \subset M(w\circ\lambda)\otimes L(\mu).
 \end{equation*}
 On the other hand, we can rewrite $F_{w,\lambda+\mu}(v_\lambda\otimes u_\mu)$
 in the following form:
 \begin{equation*}
  F_{w,\lambda+\mu}(v_\lambda\otimes u_\mu)
  = (F_{w,\lambda+\mu}v_\lambda)\otimes u_\mu 
  + \sum_{k} a_k\otimes u_k,
 \end{equation*}
 where $\{u_k\}$ is a basis of $\bigoplus_{\nu\ne\mu}L(\mu)_\nu$
 and $a_k$'s are some elements of $M(\lambda)$.
 Therefore $F_{w,\lambda+\mu}v_\lambda$ and $a_k$'s belong to 
 $M(w\circ\lambda)=U_-F_{w,\lambda}v_\lambda$.
 In particular, $F_{w,\lambda+\mu}\in U_- F_{w,\lambda}$.
 \qed
\end{proof}

\begin{remark}
 The classical limit of \theoremref{theorem:regularity-KM}
 gives another proof of the regularity of the classical $\tau$-functions
 (Theorem 1.3 of \cite{NY0012028}). 
 In \cite{NY0012028}, Noumi and Yamada proves the regularity of
 the classical $\tau$-functions by using the idea of the Sato theory of soliton equations
 \cite{Sato-Sato}. Our method is completely different from it.
 \qed 
\end{remark}


\subsection{Proof of regularity in the $q$-difference case}
\label{sec:reg-q}

In this subsection, we fix an arbitrary complex number $\hbar$ which is not a root of unity.
In order to prove the regularity of the quantum $\tau$-functions
in the $q$-difference case, it is enough to show 
the regularity in the case where $q$ is specialized at $e^\hbar$.

Let $U_\hbar(\g)$ be the associative algebra over $\C$ generated by 
$\{\,e_i,f_i,q^\beta\mid i\in I, \beta\in\Qv \,\}$ with the same fundamental
relations of $U_q(\g)$ specialized at $q=e^\hbar$. 
For each $\lambda\in P$, the Verma module $M_\hbar(\lambda)$ over $U_\hbar(\g)$ and 
its simple quotient $L_\hbar(\lambda)$ are similarly defined 
as in the case of $U_q(\g)$.
The simple module $L_\hbar(\lambda)$ is integrable if and only if $\lambda\in P_+$.
For $\lambda\in P_+$ and $w\in W$, there exists a non-zero $U_\hbar(\g)$-homomorphism 
from $M_\hbar(w\circ\lambda)$ to $M_\hbar(\lambda)$,  
which is injective and unique up to scalar multiples.
For the uniqueness, see 4.4.15 of \cite{Jos-1995}.
Although it deal with the case where $q$ is an indeterminate,
its proof holds for $q=e^\hbar\in\C^\times$ not a root of unity
(Section 2.4 of \cite{HK-2007}).
We regard $M_\hbar(w\circ\lambda)$ as a submodule of $M_\hbar(\lambda)$.

In the following, we assume that $\lambda,\mu\in P_+$ and $w\in W$.

Since the analogue of \propref{prop:reduce} for $U_\hbar(\g)$ also holds,
the regularity of the quantum $\tau$-functions in the $q$-difference case
specialized at $q=e^\hbar$ can be also obtained by the same argument 
as in \secref{sec:reg-KM} 
if there exists $U_\hbar(\g)$-modules $T_\hbar(\lambda)$ and $T_\hbar(w\circ\lambda)$
satisfying the following conditions:
\begin{enumerate}
\item[(a)] $T_\hbar(\lambda)$ is a submodule 
  of $M_\hbar(\lambda)\otimes L_\hbar(\mu)$
  and isomorphic to $M_\hbar(\lambda+\mu)$.
\item[(b)] $T_\hbar(w\circ\lambda)$ is a submodule 
  of $M_\hbar(w\circ\lambda)\otimes L_\hbar(\mu)$
  and isomorphic to $M_\hbar(w\circ(\lambda+\mu))$.
\item[(c)] $T_\hbar(w\circ\lambda)$ is a submodule of $T_\hbar(\lambda)$.
\end{enumerate}
Under the conditions above, replacing $T(M(\lambda))$ and $T(M(w\circ\lambda))$
in the proof of \theoremref{theorem:regularity-KM} 
with $T_\hbar(\lambda)$ and $T_\hbar(w\circ\lambda)$ respectively, 
we obtain the regularity of the $q$-difference version of 
the quantum $\tau$-function $\tau_{(w(\mu))}=w(\tau^\mu)$ specialized at $q=e^\hbar$.

A left $U_\hbar(\g)$-module $M$ is said to be integrally $\h$-diagonalizable 
if $M=\bigoplus_{\nu\in P}M_\nu$ where the weight subspaces are defined 
by $M_\nu=\{\,v\in M\mid q^\beta v=e^{\hbar\bra\beta,\nu\ket}v\ (\beta\in\Qv) \,\}$
for $\nu\in P$.
Let $\Ointh$ be the category of left $U_\hbar(\g)$-modules $M$ satisfying 
the conditions (A) and (B) in \secref{sec:reg-KM}.
Then, for each $\lambda\in P$, the Verma module $M_\hbar(\lambda)$ over $U_\hbar(\g)$ 
and its simple quotient $L_\hbar(\lambda)$ are objects of $\Ointh$.

Let $\{g_k\}$ and $\{g^k\}$ be the dual bases 
of the symmetrizable Kac-Moody algebra $\g$ with respect to 
the canonical non-degenerate symmetric invariant bilinear form on $\g$.
Set $\Omega = \sum_k g_k\otimes g^k$. 
Following Drinfeld \cite{Drinfeld}, using the associator, 
we can define on $\Oint$ the structure of a braided tensor category 
with braiding $e^{\hbar\Omega}$.
(For details, see \cite{Drinfeld} and \cite{EK-I}.)
Since $\hbar$ is not a root of unity, 
the universal $R$-matrix $\mathcal{R}$ for $U_\hbar(\g)$ is well-defined and 
the category $\Ointh$ is also a braided tensor category with 
braiding defined by $\mathcal{R}$.

By Theorem 4.10 of \cite{EK-VI},
there exists a braided tensor functor $F_\hbar:\Oint\to\Ointh$,
which is the identity functor at the level of $P$-graded vector spaces
and preserves the Verma modules and the integrable simple modules.
In particular, we have
\begin{align*}
 &
 F_\hbar(M(\lambda)) = M_\hbar(\lambda), \quad
 F_\hbar(M(w\circ\lambda)) = M_\hbar(w\circ\lambda), \quad
 F_\hbar(L(\mu)) = L_\hbar(\mu),
 \\ &
 F_\hbar(M(\lambda)\otimes L(\mu)) = M_\hbar(\lambda)\otimes L_\hbar(\mu), \quad
 F_\hbar(M(w\circ\lambda)\otimes L(\mu)) = M_\hbar(w\circ\lambda)\otimes L_\hbar(\mu),
 \\ &
 F_\hbar(M(\lambda+\mu)) = M_\hbar(\lambda+\mu), \quad
 F_\hbar(M(w\circ(\lambda+\mu))) = M_\hbar(w\circ(\lambda+\mu)).
\end{align*}
Define the $U_\hbar(\g)$-modules 
$T_\hbar(\lambda)$ and $T_\hbar(w\circ\lambda)$ by
\begin{equation*}
 T_\hbar(\lambda) = F_\hbar(T_\lambda^{\lambda+\mu}(M(\lambda))), \quad
 T_\hbar(w\circ\lambda) = F_\hbar(T_\lambda^{\lambda+\mu}(M(w\circ\lambda))).
\end{equation*}
Then the conditions (a), (b), and (c) are satisfied.
It concludes the following theorem.

\begin{theorem}
\label{theorem:regularity-q}
 Also in the $q$-difference case, 
 the all quantum $\tau$-functions $\tau_{(\nu)}$ ($\nu\in WP_+$)
 are regular. 
 More precisely they are polynomials in $\{f_i,q^{\pm\av_i}\}_{i\in I}$.
 \qed
\end{theorem}

This theorem is not only a canonically quantized version of Theorem 1.3 of \cite{NY0012028} 
but also its $q$-difference analogue. 


\end{document}